\documentclass[11pt]{amsart}
\usepackage[a4paper,margin=1in]{geometry}
\usepackage{amsmath,amssymb,amsthm,mathtools}
\usepackage{enumitem}
\usepackage{xcolor}
\usepackage[hidelinks]{hyperref}
\usepackage{microtype}
\usepackage[T1]{fontenc}
\usepackage[utf8]{inputenc}

\newtheorem{theorem}{Theorem}[section]
\newtheorem{proposition}[theorem]{Proposition}
\newtheorem{lemma}[theorem]{Lemma}
\newtheorem{corollary}[theorem]{Corollary}

\newtheorem{remark}[theorem]{Remark}

\newcommand{\R}{\mathbb R}
\newcommand{\Q}{\mathbb Q}
\newcommand{\N}{\mathbb N}
\newcommand{\MS}{\mathbb M}

\newcommand{\cf}[1]{[0;#1]}
\newcommand{\Kfun}{\mathcal K}

\title{The Dirichlet spectrum with respect to $L_1$ norm is $\left[\frac12,1\right]$}
\author[Nikita Shulga]{Nikita Shulga}
\address{Nikita Shulga, Sydney Mathematical Research Institute, The University of Sydney, NSW, Australia}
\email{nikita.shulga@sydney.edu.au}

\begin{document}

\begin{abstract}
We prove that the one-dimensional Dirichlet spectrum with respect to approximation in $L_1$ norm $\mathbb{D}^{[1]}$ satisfies
$$
\mathbb{D}^{[1]}=\left[\frac12,1\right].
$$
This is equivalent to the fact that the Minkowski spectrum $\MS$, associated with the Minkowski diagonal continued fraction, satisfies
$$
        \MS=\left[\frac14,\frac12\right].
$$ 
Further, we show that level sets 
$$
        \Theta_m=\{\alpha\in(0,1)\setminus\Q:\mathfrak m(\alpha)=m\},
$$
where $\mathfrak{m}(\alpha)$ is the Minkowski constant of $\alpha$, have Hausdorff dimension strictly greater than $1/2$ for any $m\in(1/4,1/2]$, while $\dim_H \Theta_{1/4}=\frac{1}{2}$.

\end{abstract}

\maketitle

\section{Introduction}
Let $F$ be a norm on $\R^2$ with closed unit ball $B_F$, and, for $t>0$ and $\alpha\in\R$, put
$$
        \Lambda_\alpha(t)
        =
        \begin{pmatrix} t^{-1} & 0 \\ 0 & t \end{pmatrix}
        \begin{pmatrix} 1 & 0 \\ -\alpha & 1 \end{pmatrix}\mathbb Z^2.
$$
If $\lambda_1(\Lambda_\alpha(t),B_F)$ denotes the first successive
minimum of this lattice with respect to $B_F$, then the corresponding
Dirichlet constant of $\alpha$ with respect to $F$ is defined as
$$
        d_F(\alpha)
        :=\limsup_{t\to\infty}\lambda_1(\Lambda_\alpha(t),B_F).
$$
Equivalently, $d_F(\alpha)$ measures the largest number with the property that if $c > d_F(\alpha)$, then for every sufficiently large $t$ there are $p, q \in\mathbb Z$ with $q > 0$ such that
\begin{equation}\label{ineq:norm}
 F(t^{-1}q,t(p-\alpha q)) <c
\end{equation}
while for $c < d_F (\alpha)$ there are arbitrarily large $t$ for which no such $p, q$ exist.  

For the unit disk $\mathcal{B}_F$ we consider its {\it critical determinant}
  $$
  \Delta_F = \inf \{ {\rm det} \,\Lambda: \,\text{there are no non-zero points of $\Lambda$ inside }\, \mathcal{B}_F\}.
  $$
One can also consider the analogous {\it normalized Dirichlet constant}  $\delta_F(\alpha)$. It can be defined as $ \delta_F(\alpha) =\Delta_F  \cdot (   d_F  (\alpha))^2$. Andersen--Duke showed that the normalized Dirichlet constant always satisfies $\delta_F(\alpha)\leq1$ and that equality holds for almost all irrationals $\alpha$ for any strongly symmetric norm $F$; this fact can be considered as a generalization of Dirichlet's theorem on other norms. This form of Dirichlet's theorem for other norms goes back to Hermite \cite{her}, who considered the $L_2$ norm, that is, the Euclidean norm. For more context, see \cite{AD} and references therein. 

Irrationals $\alpha$ with $\delta_F(\alpha)<1$ are called Dirichlet Improvable numbers with respect to norm $F$. Questions of Dirichlet improvability (or non-improvability) were recently extensively studied in various settings: in one dimension for sup-norm \cite{MR3767339, MR3798609}, for vectors in sup-norm \cite{MR4732648}, for matrices in Euclidean and arbitrary norms \cite{MR4410764,KR} and other related settings; see, for example, \cite{MR4500198,MR4587901}.

The Dirichlet spectrum with respect to the norm $F$ can thus be defined as
$$
\mathbb{D}_F:= \{ \delta_{F}(\alpha): \alpha\in\R\setminus\Q\} = \{ \Delta_F\cdot(d_{F}(\alpha))^2: \alpha\in\R\setminus\Q\},
$$

In particular, for the $L_p$ norm defined by 
$$
  F^{[p]}(x,y)=(|x|^{p}+|y|^{p})^{1/p},
$$
we write $\mathbb{D}^{[p]}$ for the Dirichlet spectrum with respect to $L_p$ norm. In particular, the classical Dirichlet spectrum (see, for example, \cite{MR4934021,MR1574660} and others) is equal to 
$
\mathbb{D}:=\mathbb{D}^{[\infty]}
$, and the corresponding critical determinant $\Delta_{F^{[\infty]}}$ for the sup-norm is equal to $1$.

Now, consider the $L_1$ norm
\begin{equation}\label{L1norm}
        F^{[1]}(x,y)=|x|+|y|.
\end{equation}

It turns out that the approximation with respect to the $L_1$ norm is connected with so-called {\it diagonal continued fraction} (DCF) introduced and studied by Minkowski \cite{MR1511108, MR1508923,MR249269}, as well as other authors who studied it either as an independent object \cite{KRAAIKAMP1989197}, or as a part of wider families of continued fractions expansions, such as $\mathcal S$-expansions (for more information, one can refer to Subsection 4.3.2 devoted to this expansion from the book \cite{MR1960327}).

Moshchevitin \cite{MR3025470} independently considered the spectrum $\MS$, called Minkowski spectrum, which can be shown to satisfy $\MS = \frac12 \mathbb{D}^{[1]}$. This fact was stated without explanation in \cite{moshchevitin2025dirichletimprovabilitylpnorms}; we include an explanation in Subsection \ref{subsec:minkspec} together with some related results on $\MS$ we will utilize in this paper.

The main result of this paper is the following statement.
\begin{theorem}\label{thm:main}
Dirichlet spectrum for $L_1$ norm satisfies 
$$\mathbb{D}^{[1]} = \left[\frac12,1\right].$$
Equivalently, the Minkowski spectrum $\MS$ satisfies $\MS=\left[\frac14,\frac12\right].$
\end{theorem}
\begin{remark}
    In some sense, Theorem \ref{thm:main} claims even more than just "Dirichlet spectrum for $L_1$ norm is an interval". From the work \cite{AD} of Andersen and Duke, we know that for any symmetric norm $F$, and any irrational $\alpha$, one has $1/2 \leq \delta_F(\alpha)\leq1$. Therefore, Theorem \ref{thm:main} says that for $L_1$ norm, the corresponding Dirichlet spectrum occupies the largest theoretically possible interval for any symmetric norm $F$.
\end{remark}
To the best of our knowledge, this is the first one-dimensional diophantine spectrum with a fully described structure. 
Recently, in a breakthrough paper \cite{agin2024dirichletspectrum}, it was proven that the multi-dimensional Dirichlet spectrum (excluding the one-dimensional case, where the structure is much more complicated) is an interval. For other work on multi-dimensional Dirichlet spectrum, see, for example, \cite{MR4618228, MR4662891}.

It turns out that we can also use the construction from the proof of Theorem \ref{thm:main} to show that the level sets with a fixed value $\mathfrak m(\alpha)=m$ are large in Hausdorff dimension. 

Define
$$
        \Theta_m= \{\alpha\in(0,1)\setminus\Q: \delta_{F^{[1]}}(\alpha)=2m\} =\{\alpha\in(0,1)\setminus\Q:\mathfrak m(\alpha)=m\}.
$$
It follows from \cite{AD} that $\Theta_{1/2}$ has full Lebesgue measure in $(0,1)$, that is, for almost every irrational $\alpha\in(0,1)$ one has $\delta_{F^{[1]}}(\alpha)=1$ (or, equivalently, $\mathfrak{m}(\alpha) = 1/2$). It also follows from \cite{MR3025470} that $\dim_H \Theta_{1/4} \geq 1/2$. Thus, we may only focus on $m\in(1/4,1/2)$.

Let
$$
        E_A=\{[0;a_1,a_2,\ldots]: a_i\ge A\text{ for all }i\},
        \qquad
        \delta_A=\dim_{H}E_A.
$$
\begin{theorem}\label{thm:large-digit-dimension}
Let $m\in(1/4,1/2)$. If $s=4m-1$ and $A$ is chosen so large that
$$
        \frac{2}{A-1}+\frac{2}{(A-1)^2}<s,
$$
then 
$$
        \dim_{H}\Theta_m\ge \delta_A.
$$
In particular, $\dim_H \Theta_m >1/2$.
\end{theorem}

\begin{remark}
The inequality $\dim_{H}\Theta_m\ge \delta_A$ implies $\dim_H \Theta_m >1/2$ because for every $A\in\N$, Good's classical theorem \cite{Good1941} gives $\delta_A>1/2$. A very strong asymptotic result for the dimension $\delta_A$ was recently obtained in \cite{MR4540838}. Also note that for $m=1/4$, we have $s=0$, and so there is no finite $A$ satisfying the assumption of Theorem \ref{thm:large-digit-dimension}.
\end{remark}

In Theorem \ref{thm:upper-bound-nontrivial} below we show that also for any $m\in(1/4,1/2)$, one has $\dim_H \Theta_m <1$. In Theorem \ref{thm:endpoint-quarter} we show that $\dim_H \Theta_{1/4} = 1/2$. For other small relevant dimensional results, see Subsection \ref{sec:dimension}.

We prove Theorem \ref{thm:main} by utilizing the recently discovered decomposition of real numbers into a sum of two regular continued fractions with partial quotients uniformly bounded from below, introduced in \cite{ShulgaCusick} and further studied by Gayfulin and Nesharim in \cite{gayfulin2025realnumbersumreal}. In short, for any irrational $s\in(0,1)$, this explicit algorithm allows to decompose 
$$
        s=\beta+\gamma,
        \qquad
        \beta=[0;b_1,b_2,\ldots],
        \qquad
        \gamma=[0;c_1,c_2,\ldots]
$$
in such a way that $c_i,b_i\to\infty$ as $i\to\infty$ along with some other useful properties of the sequences $\{c_i\}_i,\{b_i\}_i$.

Using this algorithm and its properties, for any $1/4<m<1/2$, we explicitly construct a number $\alpha$ with the property that $\mathfrak{m}(\alpha)=m$. In fact, the construction of $\alpha$ gives us enough freedom in choosing auxiliary blocks to prove Theorem \ref{thm:large-digit-dimension}.

\section{Preliminaries}
This section contains some necessary context and a reduction of the problem to something more manageable. 

First, in Subsection \ref{subsec:minkspec} we provide some information about the Minkowski diagonal continued fraction and Minkowski spectrum. This will help us to reformulate the problem about $\mathbb{D}^{[1]}$ spectrum to a problem about $\MS$ spectrum, for which there is one useful developed tool to work with. 

In Subsection \ref{subsec:change-of-var}, we slightly change the object being analyzed from the function $F$ defined in the first subsection to the newly introduced functions $H$ and $\Kfun$.

In Subsection \ref{subsec:algorithm}, we recall the continued fraction algorithm defined in \cite{ShulgaCusick} and state some necessary properties of it proven in \cite{gayfulin2025realnumbersumreal, ShulgaCusick}.

\subsection{Minkowski spectrum and diagonal continued fraction}\label{subsec:minkspec} 
For $\alpha\in\R\setminus\Q$, consider its {\it irrationality measure function} 
$$
\psi_\alpha(t) = \min_{\substack {1\leq q \leq t \\ q\in\mathbb{Z}}} \| q\alpha\|,
$$
where $\|\cdot\|$ is the distance to the nearest integer. Dirichlet's theorem is equivalent to the inequality \begin{equation}\label{eq:dirichlet}
t\psi_\alpha(t)< 1
\end{equation}for all irrational $\alpha$ and all $t\geq1$.

The claim that for all irrational $\alpha$, there are infinitely many solutions $p/q$ to $\left| \alpha - \frac{p}{q}\right|< \frac{1}{\sqrt5q^2}$ is equivalent to the claim that for any $\alpha$ there exists a sequence $\{t_n\}_{n\geq1}$ such that 
\begin{equation}\label{eq:corollary:dirichlet}
    t_n\psi_\alpha(t_n)<\frac{1}{\sqrt5}
    \end{equation}
for all $n$.

For each fixed irrational $\alpha$, the question of the optimality of the constants on the right-hand side of \eqref{eq:dirichlet} and \eqref{eq:corollary:dirichlet} is connected with finding Dirichlet and Lagrange constants of $\alpha$, respectively.  The Dirichlet and Lagrange constants of an irrational number $\alpha$ can be respectively defined as
$$
d(\alpha) := \limsup_{t\to\infty} t\psi_\alpha(t)
\quad \text{and} \quad
\lambda(\alpha) := \liminf_{t\to\infty} t\psi_\alpha(t).
$$
Consequently, the classical Dirichlet and Lagrange spectra are respectively defined as
$$
\mathbb{D}= \{ d(\alpha): \alpha\in\R\setminus\Q\}
\quad
\text{and}
\quad
\mathbb{L}= \{ \lambda(\alpha): \alpha\in\R\setminus\Q\}.
$$
These spectra were studied extensively, and, despite some significant progress, their full structure is unknown. It is known that both of them have a discrete part (that is, a countable sequence of isolated points), and that both of them have a continuous part, which is commonly referred to as a Hall's ray.

Another classical approximation result is Legendre's theorem, which says that if
\begin{equation}\label{eq:legendre}
\left| \alpha - \frac{p}{q} \right| < \frac{1}{2q^2},
\end{equation}
then $p/q$ is a continued fraction convergent to $\alpha$. The converse is not true, though it is also known that among any two consecutive convergents to $\alpha$, at least one satisfies the inequality \eqref{eq:legendre}. Minkowski \cite{MR1511108} considered a special continued fraction, which is now called {\it Minkowski diagonal continued fractions}, whose convergents are exclusively the ones satisfying the inequality \eqref{eq:legendre}.

Moshchevitin introduced the function $\mu_\alpha(t)$ associated with the
Minkowski diagonal continued fraction as follows. Fix irrational $\alpha$ and let 
\begin{equation}\label{sequence:conv}
Q_0<Q_1<\cdots<Q_{n}<Q_{n+1}<\cdots
\end{equation}
be the subsequence of denominators of convergents $A/Q$ to an irrational $\alpha$ which
satisfy Legendre's inequality
\begin{equation}\label{ineq:legendre}
        \left|\alpha-\frac{A}{Q}\right|<\frac{1}{2Q^2}.
\end{equation}
The function $\mu_\alpha(t)$ is obtained by linearly interpolating the points
$(Q_n,\|Q_n\alpha\|)$. To be precise, the function $\mu_\alpha(t)$ is defined as
$$
\mu_\alpha(t) = \frac{Q_{n+1}-t}{Q_{n+1}-Q_n}\|Q_n\alpha\| +\frac{t-Q_n}{Q_{n+1}-Q_n}\|Q_{n+1}\alpha\|, \quad  Q_n\leq t \leq Q_{n+1}.
$$
This function was studied as a standalone object in, for example, \cite{10.1063/1.3630039}.

Note that the regular Lagrange constant can be defined as 
$$\lambda(\alpha)= \liminf_{t\to+\infty}t\cdot\mu_\alpha(t).$$

The corresponding limsup value is
$$
        \mathfrak{m}(\alpha):=\limsup_{t\to\infty}t\mu_\alpha(t),
$$
and the Minkowski spectrum is
$$
        \MS=\{\mathfrak{m}(\alpha):\alpha\in\R\setminus\Q\}.
$$
Moshchevitin \cite{MR3025470} proved that $\MS\subset[1/4,1/2]$ and showed that both $1/4$ and $1/2$ are elements of $\MS$.

For irrational $\alpha$, in \cite{moshchevitin2025dirichletimprovabilitylpnorms}, the connection between Minkowski constant $\mathfrak m(\alpha)$ and (non-normalized) Dirichlet constant for $L_1$ norm $d^{[1]}(\alpha)$ was given as
\begin{equation}\label{eq:mink:dir}
        \mathfrak m(\alpha)=\frac{(d^{[1]}(\alpha))^2}{4}.
\end{equation}
Thus, the Minkowski spectrum may be introduced as the Dirichlet spectrum for one-dimensional rational approximation with respect to the $L_1$ norm, with a different normalizing factor:
$$
        \MS=
        \left\{
        \frac{(d^{[1]}(\alpha))^2}{4}:\alpha\in\R\setminus\Q
        \right\} = \frac12 \mathbb{D}^{[1]}.
$$
The continued-fraction description of the same quantity, in terms of Minkowski's diagonal continued fraction, is recalled below as it is needed for the proof.

Here,  we provide a justification of the equation \eqref{eq:mink:dir}, as there is none present in the literature, as we believe.
\begin{proposition}
We have 
$$
 \mathfrak m(\alpha)=\frac{(d^{[1]}(\alpha))^2}{4}.
$$
\end{proposition}
\begin{proof}
 Put $E_n=\|Q_n\alpha\|$. By the definition of $\Lambda_\alpha(t)$ and of the $L_1$ norm, for all sufficiently large $t$ one has
$$
\lambda_1(\Lambda_\alpha(t),B_{F^{[1]}})
=
\min_{q\geq 1}
\left(\frac{q}{t}+t\|q\alpha\|\right).
$$
Indeed, a lattice vector corresponding to $(q,p)\in\mathbb Z^2$ has $L_1$ length $|q|/t+t|p-\alpha q|$, and for fixed $q\neq 0$ the best choice of $p$ gives $|p-\alpha q|=\|q\alpha\|$. The vectors with $q=0$ have length at least $t$, and hence do not contribute to the first minimum for large $t$.

By the construction of the Minkowski diagonal continued fraction, the points $(Q_n,E_n)$ are precisely the vertices of the lower convex hull of the set of points $(q,|q\alpha|)$, $q\geq 1$. Equivalently, the piecewise-linear function $\mu_\alpha$ joining the points $(Q_n,E_n)$ is the lower convex envelope of this set. Therefore, replacing the discrete set with this broken line does not change the above minimum:
$$
\lambda_1(\Lambda_\alpha(t),B_{F^{[1]}})
=
\min_n
\left(\frac{Q_n}{t}+tE_n\right).
$$
Indeed, all non-vertex points lie on or above the broken line, and on each linear segment the function $x/t+t\mu_\alpha(x)$ is affine in $x$, so its minimum is attained at an endpoint.

Now consider one segment of the graph of $\mu_\alpha$, say over $[Q_n,Q_{n+1}]$. Write it as $\mu_\alpha(x)=A_n-B_nx$, where $B_n>0$. The two adjacent expressions $Q_n/t+tE_n$ and $Q_{n+1}/t+tE_{n+1}$ are equal exactly when $t^2=1/B_n$, and at this value of $t$ their common value is $A_n/\sqrt{B_n}$. Hence the corresponding local maximum of $\lambda_1(\Lambda_\alpha(t),B_{F^{[1]}})$ is $A_n/\sqrt{B_n}$.

On the other hand, on the same segment,
$$
x\mu_\alpha(x)=A_nx-B_nx^2.
$$
Its maximum is attained at $x=A_n/(2B_n)$ and is equal to $A_n^2/(4B_n)$. For the segments of the Minkowski diagonal broken line, this point lies in the segment: this is exactly the same Legendre condition $Q E<1/2$ which selects the Minkowski diagonal convergents, with the analogous statement for a segment passing over one skipped regular convergent. Thus, each segment contributes
$$
\max_{Q_n\leq x\leq Q_{n+1}}x\mu_\alpha(x)
=
\frac14
\left(\frac{A_n}{\sqrt{B_n}}\right)^2.
$$
Taking the limsup over all segments gives
$$
\mathfrak m(\alpha)
=
\limsup_{x\to\infty}x\mu_\alpha(x)
=
\frac14
\left(
\limsup_{t\to\infty}
\lambda_1(\Lambda_\alpha(t),B_{F^{[1]}})
\right)^2
=
\frac{(d^{[1]}(\alpha))^2}{4}.
$$
This proves \eqref{eq:mink:dir}.

Since $\Delta_{F^{[1]}}=1/2$, we also get
$$
\delta_{F^{[1]}}(\alpha)
=
\Delta_{F^{[1]}}(d^{[1]}(\alpha))^2
=
2\mathfrak m(\alpha).
$$
Therefore $\MS=\frac12\mathbb D^{[1]}$.

\end{proof}

In \cite{MR3025470}, Moshchevitin showed that $\MS\subset[1/4,1/2]$ and showed that both $1/4$ and $1/2$ are elements
of $\MS$. In the same paper it was shown that the quantity $\mathfrak{m}(\alpha)$ is described by values of two functions
$$
        G(x,y)=\frac{1+x+y}{4},
        \qquad
        F(x,y)=\frac{(1-xy)^2}{4(1+xy)(1-x)(1-y)}.
$$
The precise theorem is stated as follows.
\begin{theorem}[Theorem 1, \cite{MR3025470}]\label{thm:mosh:FG}
Let $\{q_n\}_n$ be the non-decreasing sequence of denominators convergents to the regular continued fraction expansion of $\alpha=[0;a_1,\ldots,a_n,\ldots]$. Denote $\alpha^*_\nu=[0;a_\nu,\ldots,a_1]$ and $\alpha^{-1}_{\nu}=[0;a_\nu,a_{\nu+1},\ldots]$. Let $\{Q_n\}_n$ be the sequence defined by \eqref{sequence:conv}, \eqref{ineq:legendre}.
Put
\begin{align*}
\mathfrak{m}_n(\alpha) =  \begin{cases} 
G(\alpha_\nu^*, \alpha_{\nu+2}^{-1}) \quad &\text{if } (Q_n,Q_{n+1})=(q_{\nu-1},q_{\nu+1}) \text{ with some } \nu,  \\  
F(\alpha_{\nu+1}^*, \alpha_{\nu+2}^{-1}) \quad &\text{if } (Q_n,Q_{n+1})=(q_{\nu},q_{\nu+1})  \text{ with some } \nu.
 \end{cases}
\end{align*}
Then
$$
\mathfrak{m}(\alpha) = \limsup_{n\to\infty} \mathfrak{m}_n(\alpha).
$$
\end{theorem}

To reiterate, the term $F$ occurs when two consecutive convergents satisfy Legendre's inequality \eqref{eq:legendre}, while $G$ occurs when one convergent does not satisfy \eqref{eq:legendre}, and so is skipped when constructing a diagonal continued fraction.  

Theorem \ref{thm:mosh:FG} provides us with a powerful tool for analyzing the $\mathfrak m(\alpha)$, and hence the $\delta_{F^{[1]}}(\alpha)$. From now on, we treat the $F,G$ definition of Minkowski constant and Dirichlet constant with respect to $L_1$ as the paramount one and work with it directly in our proofs.

\subsection{Change of function and change of variables}\label{subsec:change-of-var}
To easier deal with $F$--values, which is the essence of Step 11, we prefer to introduce another function $\Kfun$ that is slightly better behaved.

For $0<x,y<1$, put
$$
        H(x,y)=4F(x,y)-1
        =\frac{x^2y+xy^2-4xy+x+y}{(1-x)(1-y)(1+xy)}.
$$
Thus
\begin{equation}\label{F_to_H}
        F(x,y)\le \frac{1+s}{4}
        \quad\Longleftrightarrow\quad
        H(x,y)\le s.
\end{equation}
We shall use the following elementary change of variables.

\begin{lemma}\label{lem:K}
Let
$$
        u=\frac{x}{1-x},\qquad v=\frac{y}{1-y}.
$$
Then
$$
        H(x,y)=\Kfun(u,v),
$$
where
$$
        \Kfun(u,v)=
        \frac{u^2+u+v^2+v}{2uv+u+v+1}.
$$

\end{lemma}

\begin{proof}
The identity follows by substituting
$x=u/(1+u)$ and $y=v/(1+v)$ into the displayed formula for $H$.
\end{proof}

\begin{lemma}\label{lem:endpoint-principle}
Let $I,J\subset[0,\infty)$ be compact intervals.  The maximum of
$\Kfun$ on $I\times J$ is attained at one of the four corners.
\end{lemma}

\begin{proof}
For fixed $v\ge0$,
$$
        \frac{\partial \Kfun}{\partial u}(u,v)
        =
        \frac{(u+v+1)((2v+1)u-2v^2-v+1)}
             {(2uv+u+v+1)^2}.
$$
The denominator and the factor $u+v+1$ are positive, while the remaining
factor is affine increasing in $u$.  Hence, as a function of $u$,
$\Kfun(u,v)$ is either monotone on $I$ or decreases and then increases; in
both cases its maximum on $I$ is attained at an endpoint of $I$.  Applying
this first in the $u$-variable gives
$$
        \max_{(u,v)\in I\times J}\Kfun(u,v)
        =
        \max\Big\{\max_{v\in J}\Kfun(u_0,v),
                 \max_{v\in J}\Kfun(u_1,v)\Big\},
$$
where $u_0,u_1$ are the endpoints of $I$.  The same argument in the
$v$-variable, using
$$
        \frac{\partial \Kfun}{\partial v}(u,v)
        =
        \frac{(u+v+1)((2u+1)v-2u^2-u+1)}
             {(2uv+u+v+1)^2},
$$
shows that each of the two inner maxima is attained at an endpoint of $J$.
Thus, a maximum on $I\times J$ is attained at a corner.
\end{proof}

\subsection{Sum decomposition and its properties}\label{subsec:algorithm}
Here, we recall the definition of the algorithm introduced by the author in \cite{ShulgaCusick} and mention some of its properties that will be used while proving Theorem \ref{thm:main}. We note that the notation in \cite{ShulgaCusick} is different from \cite{gayfulin2025realnumbersumreal}. Namely, the roles of $b_i$'s and $c_i$'s are reversed. In the following, we use the notation from \cite{gayfulin2025realnumbersumreal}.

Let $s\in(0,1)$ be the target number which we want to expand into a sum of two continued fractions with partial quotients bounded from below by $2$. Note that if one wants to get a larger uniform lower bound on partial quotients, say $k\geq3$, then one can take any $s\in(0,\frac{1}{k-1})$ and run the same algorithm. For more details, see \cite{ShulgaCusick}.

Let $a_i(x)$ denote the function that returns the $i$th partial quotient of the number $x$. Also, we write $\frac{p_n}{q_n}=[0;b_1,\ldots,b_n]$ and $\frac{s_n}{t_n}=[0;c_1,\ldots,c_n]$. Set $b_1  = a_1(s)+1$ and for $n\geq1$ iteratively define
\begin{equation}\label{algorithm:bn}
c_n = a_n \left(    s - \frac{p_n}{q_n}  \right)
\end{equation}
and
\begin{equation}\label{algorithm:cn}
b_{n+1} = a_{n+1} \left(    s - \frac{s_n}{t_n}  \right)  +1.
\end{equation}

This algorithm produces
$$
        \beta=\cf{b_1,b_2,\ldots},
        \qquad
        \gamma=\cf{c_1,c_2,\ldots},
        \qquad
        s=\beta+\gamma,
$$
in the infinite case.  In the finite case, we still have $s=\beta+\gamma$, and also both continued fractions of $\beta$ and $\gamma$ will be finite.

We will use the following properties of the algorithm defined above.
\begin{lemma}\label{lem:first-second}
Assume that $b_2$ from the algorithm \eqref{algorithm:bn}, \eqref{algorithm:cn} exists.  Then
\begin{equation}\label{eq:first-step}
        c_1\geq b_1(b_1-1),
\end{equation}
and
\begin{equation}\label{eq:second-step}
        b_2-1>
        \frac{c_1(c_1+1)-b_1-b_1^2}{b_1^2}.
\end{equation}
Also, the number $s$ is an element of
\begin{equation}\label{def:C1}
       C_1=\bigl(\cf{b_1}+\cf{c_1+1},\cf{b_1}+\cf{c_1}\bigr].
\end{equation}
 In particular, $\frac{1}{b_1}+\frac{1}{c_1+1}<s$.
\end{lemma}
\begin{proof}
    A trivial consequence of the definition of the algorithm.
\end{proof}

We also use the growth estimates
$$
        b_{n+1}\ge b_n,
        \qquad
        b_{n+2}\ge b_n+1,
        \qquad
        b_n\to\infty,
        \qquad
        c_n\to\infty
$$
in the case of irrational $s$, and the inequality
$$
        c_n>b_n\qquad(n\ge2).
$$
These were proven by Gayfulin--Nesharim in \cite{gayfulin2025realnumbersumreal}.

\section{General overview of the idea for Theorem \ref{thm:main}}
Here we briefly describe the strategy of the proof of the main theorem of this paper, which is Theorem \ref{thm:main}.

\begin{enumerate}
    \item Parametrize the interval $(\frac14,\frac12)$ with $0<s<1$ by $\Lambda(s)=\frac{1+s}{4}$.

    \item For given $0<s<1$, we will build an irrational number $\alpha:=\alpha(s)$, for which $\mathfrak{m}(\alpha)=\Lambda(s)$.

    \item If $s$ is irrational, decompose $0<s<1$ as $$
        s=\beta+\gamma,
        \qquad
        \beta=[0;b_1,b_2,\ldots],
        \qquad
        \gamma=[0;c_1,c_2,\ldots]
$$
in such a way that $c_i,b_i\geq2$ for all $i$ and $c_i,b_i\to\infty$ as $i\to\infty$. If $s$ is rational, consider a sequence of irrationals $s_j\to_{j\to\infty}s$, and decompose each $0<s_j<1$ as $s_j=\beta_j+\gamma_j$ with the same properties as before.

\item (To lighten the exposition of the overview, we focus on the case of $s$ being irrational)
Build blocks of the form 
$$
        W_N=(b_N,\ldots,b_1,1,c_1,\ldots,c_N).
$$

\item Consider two sequences of positive integers $\{N_j\}_j$ and $\{d_j\}_j$, that are strictly increasing to $\infty$ and only take large enough values for all $j$. The numbers $d_j$ are called separators.

\item Concatenate the words/blocks $W_{N_j}$ with $d_j$ with suitable indices to get an infinite word
$$
(W_{N_1},d_1,W_{N_2},d_2,W_{N_3},d_3,\ldots)
$$

\item Consider the irrational number $\alpha:=\alpha(s)$ given by its continued fraction expansion defined by the infinite word above
$$
        \alpha=[0;W_{N_1},d_1,W_{N_2},d_2,W_{N_3},d_3,\ldots].
$$
\item Note that the only convergents to $\alpha$ that do not satisfy Legendre's inequality are exactly the ones corresponding to convergents right before the partial quotient $1$ in the middle of each $W_{N_j}$.

\item Using Moshchevitin's Theorem \ref{thm:mosh:FG} on the values of Minkowski constants, conclude that we need to calculate $G$--values for convergents right before the partial quotient $1$, and calculate $F$--values for all other convergents.

\item Show that the limit of $G$--values over the sequence of convergents right before the partial quotients $1$ is equal to $\Lambda(s)=\frac{1+s}{4}$.

\item Show that for large enough indices, $F$--values for all other convergents will be not larger than $\Lambda(s)$ plus something tending to $0$. This will be possible because the sequences $b_i,c_i$ and separators $d_j$ are well-behaved. In particular, they are "large" and growing.

\item Conclude that the $\limsup$ of all $G$--values and $F$--values is exactly $\Lambda(s)$ coming from the $G$--values.

\item Note that as $0<s<1$ was arbitrary, $\Lambda(s)$ takes all values in $(1/4,1/2)$, while the endpoints $1/4$ and $1/2$ are attainable by Moshchevitin's result.
\end{enumerate}

The hardest technical part is Step 11, where we show that all other possible "cuts" of the continued fraction will lead to smaller values of $\mathfrak{m}(\alpha)$.

\section{Uniform bounds of the function $\Kfun$ on some rectangles}
This section states several lemmas that will help us to bound the values of the function $\Kfun(u,v)$ at all possible non-central cuts of a two-sided infinite word 
$$\ldots,b_n,\ldots,b_2,b_1,1,c_1,c_2,\ldots,c_n,\ldots.$$

The proofs of the next three statements are very technical checks, so in order not to interrupt the conceptual flow, we move their proofs to Appendix \ref{sec:addendix}. 

The first statement allows us to control the two central-adjacent cuts between $b_2,b_1$ and $c_1,c_2$.

\begin{lemma}\label{lem:arithmetic-bound}
Let $b_1\ge2$ and $c_1\geq b_1(b_1-1)$ be integers.  Define
$$
        L=\frac{1}{b_1}+\frac1{c_1+1}.
$$
If $(b_1,c_1)=(2,2)$, put $U=1$.  Otherwise put
$$
        U=\frac{b_1^2}{c_1(c_1+1)-b_1-b_1^2}.
$$
Also put
$$
        V_{b_1}=\frac{c_1+1}{b_1c_1+b_1-1},
        \qquad
        V_{c_1}=\frac{b_1+1}{b_1c_1+c_1-1}.
$$
Then
\begin{equation}\label{eq:arithmetic-main}
        \Kfun(u,v)\le L
\end{equation}
whenever
\begin{equation}\label{eq:arithmetic-first-range}
        0\le u\le U,
        \qquad
        \frac{1}{b_1}\le v\le V_{b_1},
\end{equation}
and also whenever
\begin{equation}\label{eq:arithmetic-second-range}
        0\le v\le U,
        \qquad
        \frac{1}{c_1}\le u\le V_{c_1}.
\end{equation}
\end{lemma}

This second lemma bounds $F$-cuts between digits $b_{i+1},b_i$ with $i\geq2$.

\begin{lemma}[Left non-adjacent rectangle]\label{lem:left-rectangle}
Let $b_1\ge2$, $c_1\geq b_1(b_1-1)$, and
$$
        L=\frac{1}{b_1}+\frac1{c_1+1}.
$$
Put
$$
        U_L=
        \begin{cases}
        \dfrac34, & (b_1,c_1)=(2,2),\\[6pt]
        \dfrac{b_1^2}{c_1(c_1+1)-b_1-b_1^2}, & (b_1,c_1)\ne(2,2).
        \end{cases}
$$
Then
$$
        \Kfun(u,v)\le L
$$
for every
$$
        0\le u\le \frac{1}{b_1},
        \qquad
        0\le v\le U_L .
$$
\end{lemma}

Third lemma bounds $F$-cuts between digits $c_i,c_{i+1}$ with $i\geq2$.

\begin{lemma}[Right non-adjacent square]\label{lem:right-square}
Let $b_1\ge2$, $c_1\geq b_1(b_1-1)$, and
$$
        L=\frac{1}{b_1}+\frac1{c_1+1}.
$$
Put
$$
        W=
        \begin{cases}
        \dfrac12, & (b_1,c_1)=(2,2),\\[6pt]
        \dfrac{b_1^2}{c_1(c_1+1)-b_1}, & (b_1,c_1)\ne(2,2).
        \end{cases}
$$
Then
$$
        \Kfun(u,v)\le L
$$
for every
$$
        0\le u\le W,
        \qquad
        0\le v\le W .
$$
\end{lemma}

\section{Cuts of the two-sided word}

Let $s\in(0,1)$ be irrational, and let the algorithm \eqref{algorithm:bn}, \eqref{algorithm:cn} give the infinite
continued fractions
$$
        \beta=[0;b_1,b_2,\ldots],
        \qquad
        \gamma=[0;c_1,c_2,\ldots],
        \qquad
        s=\beta+\gamma .
$$
Insert a central digit $1$, and consider the two-sided word
$$
        \ldots,b_3,b_2,b_1,1,c_1,c_2,c_3,\ldots .
$$
The $\mathfrak{m}_n$ value corresponding to the cut before the partial quotient $1$ is equal to
$$
        G(\beta,\gamma)=\frac{1+s}{4}=\Lambda(s).
$$
The following theorem proves that no non-central cut in this word gives a larger value. We will use the correspondence \eqref{F_to_H}, so that if we want to show that the $F$--value is not greater than $\Lambda(s)$, we can instead check that the $H$--value is not greater than $s$.

\begin{theorem}\label{thm:cut-admissibility}
For every $r\ge2$,
\begin{equation}\label{eq:left-cut}
        H\!\left([0;b_r,b_{r+1},\ldots],
        [0;b_{r-1},b_{r-2},\ldots,b_1,1,c_1,c_2,\ldots]\right)
        \le s.
\end{equation}
For every $r\ge1$,
\begin{equation}\label{eq:right-cut}
        H\!\left([0;c_r,c_{r-1},\ldots,c_1,1,b_1,b_2,\ldots],
        [0;c_{r+1},c_{r+2},\ldots]\right)
        \le s.
\end{equation}
\end{theorem}

\begin{proof}
Put
$$
        L=\frac{1}{b_1}+\frac1{c_1+1}.
$$
By the algorithm \eqref{algorithm:bn}, \eqref{algorithm:cn}, we always have $b_1\ge2$.  Since $s$ belongs to the first interval $C_1$ defined in \eqref{def:C1}, we have
\begin{equation}\label{eq:s-lower-bound}
        s\ge L.
\end{equation}
By Lemma \ref{lem:first-second},
\begin{equation}\label{eq:c1-lower-thm}
        c_1\geq b_1(b_1-1)
\end{equation}
and, if $(b_1,c_1)\ne(2,2)$,
\begin{equation}\label{eq:b2-recip-bound}
        \frac1{b_2-1}
        \le
        \frac{b_1^2}{c_1(c_1+1)-b_1-b_1^2}.
\end{equation}
In the exceptional case $(b_1,c_1)=(2,2)$, we only use the trivial inequality
$1/(b_2-1)\le1$.

First, consider the left-adjacent cut $r=2$.  Put
$$
        x=[0;b_2,b_3,\ldots],
        \qquad
        y=[0;b_1,1,c_1,c_2,\ldots],
$$
and
$$
        u=\frac{x}{1-x},
        \qquad
        v=\frac{y}{1-y}.
$$
Then
\begin{equation}\label{eq:left-adjacent-u}
        u=[0;b_2-1,b_3,\ldots]
        \le \frac1{b_2-1},
\end{equation}
and
$$
        v=[0;b_1-1,1,c_1,c_2,\ldots].
$$
If we let $z=[0;c_2,c_3,\ldots]\in(0,1)$, then
$$
        v=\frac{c_1+z+1}{b_1c_1+b_1z+b_1-1},
$$
and this expression is decreasing in $z$.  Therefore
\begin{equation}\label{eq:left-adjacent-v}
        \frac{1}{b_1}\le v\le \frac{c_1+1}{b_1c_1+b_1-1}=:V_{b_1}.
\end{equation}
Equations \eqref{eq:c1-lower-thm}--\eqref{eq:left-adjacent-v}, together with Lemma \ref{lem:arithmetic-bound}, give
$$
        H(x,y)=\Kfun(u,v)\le L\le s.
$$
This proves \eqref{eq:left-cut} for $r=2$.

Now let $r\ge3$ on the left.  With
$$
        x=[0;b_r,b_{r+1},\ldots],
        \qquad
        y=[0;b_{r-1},\ldots,b_1,1,c_1,c_2,\ldots],
$$
we have
$$
        u=\frac{x}{1-x}=[0;b_r-1,b_{r+1},\ldots],
$$
$$
        v=\frac{y}{1-y}
        =[0;b_{r-1}-1,b_{r-2},\ldots,b_1,1,c_1,c_2,\ldots].
$$
The estimates $b_{n+1}\ge b_n$ and $b_{n+2}\ge b_n+1$ imply
$$
        b_r-1\ge b_1,
        \qquad r\ge3,
$$
so
\begin{equation}\label{eq:left-nonadjacent-u}
        0<u\le \frac{1}{b_1}.
\end{equation}
If $(b_1,c_1)\ne(2,2)$, then $b_{r-1}\ge b_2$, and \eqref{eq:b2-recip-bound} gives
\begin{equation}\label{eq:left-nonadjacent-v}
        0<v\le \frac1{b_{r-1}-1}\le \frac1{b_2-1}
        \le \frac{b_1^2}{c_1(c_1+1)-b_1-b_1^2}.
\end{equation}
If $(b_1,c_1)=(2,2)$, then either the first digit of $v$ is at least $2$, in which case $v\le1/2$, or the first digit is $1$.  The latter can only occur when $r=3$ and $b_2=2$.  Then
$$
        v=[0;1,2,1,2,c_2,c_3,
        \ldots].
$$
Because $c_2>b_2=2$, we have $[0;c_2,c_3,
\ldots]\le1/3$.  Writing $z=[0;c_2,c_3,\ldots]$,
$$
        v=\frac{3z+8}{4z+11}\le \frac{27}{37}<\frac34.
$$
Thus, in all exceptional cases
\begin{equation}\label{eq:left-exceptional-v}
        0<v\le \frac34.
\end{equation}
Thus, Lemma \ref{lem:left-rectangle} applies in both cases and gives
$$
        H(x,y)=\Kfun(u,v)\le L\le s.
$$
This proves all left inequalities \eqref{eq:left-cut}.

Next, consider the right-adjacent cut $r=1$.  Put
$$
        x=[0;c_1,1,b_1,b_2,\ldots],
        \qquad
        y=[0;c_2,c_3,\ldots],
$$
and again write $u=x/(1-x)$, $v=y/(1-y)$.  Then
$$
        u=[0;c_1-1,1,b_1,b_2,\ldots].
$$
If $z=[0;b_2,b_3,\ldots]\in(0,1)$, then
$$
        u=\frac{b_1+z+1}{b_1c_1+c_1z+c_1-1},
$$
again, a decreasing function of $z$.  Hence
\begin{equation}\label{eq:right-adjacent-u}
        \frac{1}{c_1}\le u\le \frac{b_1+1}{b_1c_1+c_1-1}=:V_{c_1}.
\end{equation}
Also
$$
        v=[0;c_2-1,c_3,\ldots]
        \le \frac1{c_2-1}.
$$
Since $c_2>b_2$, we have $c_2-1\ge b_2$, and hence
\begin{equation}\label{eq:right-adjacent-v}
        v\le \frac1{b_2}\le \frac1{b_2-1}.
\end{equation}
Together with \eqref{eq:b2-recip-bound}, or with the trivial exceptional bound, Lemma
\ref{lem:arithmetic-bound} gives
$$
        H(x,y)=\Kfun(u,v)\le L\le s.
$$
This proves \eqref{eq:right-cut} for $r=1$.

Finally, take $r\ge2$ on the right.  Then
$$
        u=[0;c_r-1,c_{r-1},\ldots,c_1,1,b_1,b_2,\ldots],
$$
$$
        v=[0;c_{r+1}-1,c_{r+2},\ldots].
$$
Since $c_n>b_n$ for $n\ge2$, and since $b_n$ is non-decreasing,
\begin{equation}\label{eq:right-nonadjacent-uv}
        0<u\le \frac1{c_r-1}\le \frac1{b_r}\le \frac1{b_2},
        \qquad
        0<v\le \frac1{c_{r+1}-1}\le \frac1{b_{r+1}}
        \le \frac1{b_2}.
\end{equation}
If $(b_1,c_1)\ne(2,2)$, then Lemma \ref{lem:first-second} gives
$$
        b_2>
        \frac{c_1(c_1+1)-b_1}{b_1^2},
$$
and therefore
\begin{equation}\label{eq:right-square-bound}
        \frac1{b_2}
        \le
        \frac{b_1^2}{c_1(c_1+1)-b_1}.
\end{equation}
If $(b_1,c_1)=(2,2)$, then simply $b_2\ge2$, so $1/b_2\le1/2$.  Lemma \ref{lem:right-square}, applied with \eqref{eq:right-nonadjacent-uv} and \eqref{eq:right-square-bound}, yields
$$
        H(x,y)=\Kfun(u,v)\le L\le s.
$$
This proves \eqref{eq:right-cut} for every $r\ge2$, and the theorem follows.
\end{proof}

The next lemma shows that if we change a two-sided infinite word purely consisting of $b_i$'s and $c_i$'s to one that has a subword $(b_N,\ldots,b_1,1,c_1,\ldots,c_N)$ for some large $N$, followed by arbitrary continuations from both sides, then its $F$-values for all cuts reasonably close to the $b_1,1,c_1$, are different from $F$-cuts of an original two-sided infinite words by at most $\epsilon$ which tends to $0$ as $N\to\infty$.

\begin{lemma}
\label{lem:finite-window-stability}
Fix an irrational number $0<s=\beta+\gamma<1$, where
$$
        \beta=[0;b_1,b_2,\ldots],
        \qquad
        \gamma=[0;c_1,c_2,\ldots].
$$
Fix $R\ge 1$ and $\varepsilon>0$. Then there exists
$N_0=N_0(R,\varepsilon)$ such that for every $N\ge N_0$ the following property holds.

Consider the finite word
$$
        W_N=(b_N,\ldots,b_1,1,c_1,\ldots,c_N).
$$
If $W_N$ is embedded into an arbitrary two-sided continued-fraction word by adding arbitrary blocks of positive integers to its left and to its right, then all cuts of $W_N$ whose distance from the central digit $1$ is at most $R$ have the following stability property.

First, the two tails at the central skipped-convergent cut have the form
$$
        \beta_{N,\rho}=[0;b_1,\ldots,b_N,\rho],
        \qquad
        \gamma_{N,\sigma}=[0;c_1,\ldots,c_N,\sigma],
        \qquad 0\le \rho,\sigma\le 1,
$$
and satisfy
$$
        |\beta_{N,\rho}-\beta|<\varepsilon,
        \qquad
        |\gamma_{N,\sigma}-\gamma|<\varepsilon.
$$
In particular, for all sufficiently large $N$ one has
$$
        \beta_{N,\rho}+\gamma_{N,\sigma}<1
$$
uniformly over all choices of the external tails.

Second, every non-central cut of $W_N$ whose distance from the central digit $1$ is at most $R$ has $F$-value at most
$$
        \Lambda(s)+\varepsilon,
$$
again uniformly over all choices of the external tails.
\end{lemma}

\begin{proof}
We use the elementary uniform contraction of continued-fraction cylinders. If two continued fractions have the same first $M$ partial quotients, then their values differ by at most the length of the corresponding cylinder. In particular, this difference tends to $0$ as $M\to\infty$, uniformly in the remaining tail.

The central tails $\beta_{N,\rho}$ and $\gamma_{N,\sigma}$ have common prefixes of lengths $N$ with $\beta$ and $\gamma$, respectively. Hence
$$
        \sup_{0\le \rho\le 1}|\beta_{N,\rho}-\beta|\to 0,
        \qquad
        \sup_{0\le \sigma\le 1}|\gamma_{N,\sigma}-\gamma|\to 0
$$
as $N\to\infty$. Since $s=\beta+\gamma<1$, the inequality
$$
        \beta_{N,\rho}+\gamma_{N,\sigma}<1
$$
follows uniformly in $\rho,\sigma$ for all sufficiently large $N$.

Now, consider a non-central cut whose distance from the central digit $1$ is at most $R$. There are only finitely many such cuts. For each of them, the two one-sided continued-fraction tails determined by the finite word $W_N$ have common prefixes with the corresponding tails in the infinite two-sided word, and the lengths of these common prefixes tend to infinity with $N$, uniformly over all choices of the external continuation. Therefore, the two finite-word tails converge uniformly to the corresponding infinite-word tails.

The functions defining the local $F$-values are continuous on the relevant compact set of continued-fraction tails. Hence, the corresponding $F$-values converge uniformly to the $F$-values of the same cuts in the infinite two-sided word. By Theorem~\ref{thm:cut-admissibility}, all those limiting non-central values are at most $\Lambda(s)$. Therefore, after increasing $N_0$ if necessary, all non-central cuts at a distance at most $R$ from the central digit have an $F$-value at most
$$
        \Lambda(s)+\varepsilon.
$$
This proves the lemma.
\end{proof}

\section{The proof of Theorem \ref{thm:main}}\label{sec:proof:main}

For $0<s<1$, put
$$
        \Lambda(s)=\frac{1+s}{4}.
$$
We shall prove that $\Lambda(s)\in\MS$.  Since $\Lambda$ maps $(0,1)$ onto $(1/4,1/2)$, this will give the whole open interval.  The two endpoints are elements of $\MS$ by the result of Moshchevitin.

\begin{proof}[Proof of Theorem \ref{thm:main}]
We first treat irrational $s$.  Let
$$
        s=\beta+\gamma,
        \qquad
        \beta=[0;b_1,b_2,\ldots],
        \qquad
        \gamma=[0;c_1,c_2,\ldots]
$$
be its decomposition from the algorithm \eqref{algorithm:bn}, \eqref{algorithm:cn}.  For $N\ge1$, put
$$
        W_N=(b_N,\ldots,b_1,1,c_1,\ldots,c_N).
$$
We shall build $\alpha$ by concatenating longer and longer blocks $W_N$ and separating them by very large partial quotients.

Choose a sequence $\varepsilon_j\downarrow0$.  Choose increasing integers
$A_j\to\infty$ so large that, with $\eta_j=1/A_j$,
$$
        2\eta_j+2\eta_j^2<4\varepsilon_j.
$$
Since $b_n\to\infty$ and $c_n\to\infty$, choose increasing integers
$R_j\to\infty$ such that
$$
        b_n\ge A_j+1,
        \qquad
        c_n\ge A_j+1
        \qquad(n\ge R_j).
$$

For each $j$, apply Lemma~\ref{lem:finite-window-stability} with $R=R_j$ and accuracy $\varepsilon_j$. Choose $N_j\ge R_j$, with $N_j\to\infty$, so large that the conclusions of the lemma hold for the word
$$
        W_{N_j}=(b_{N_j},\ldots,b_1,1,c_1,\ldots,c_{N_j}).
$$
Thus, uniformly over all choices of the blocks placed before and after $W_{N_j}$, every non-central cut of $W_{N_j}$ whose distance from the central digit $1$ is at most $R_j$ has an $F$--value at most
$$
        \Lambda(s)+\varepsilon_j.
$$
Moreover, for the actual continuations which will occur after the blocks are concatenated, the central tails have the form
$$
        \beta_j=[0;b_1,b_2,\ldots,b_{N_j},\rho_j],
        \qquad
        \gamma_j=[0;c_1,c_2,\ldots,c_{N_j},\sigma_j],
$$
with $0\le \rho_j,\sigma_j\le 1$, and satisfy
$$
        \beta_j+\gamma_j<1,
        \qquad
        \left|\frac{1+\beta_j+\gamma_j}{4}-\Lambda(s)\right|<\varepsilon_j.
$$

Finally, choose separators $d_j\to\infty$ with
$$
        d_j\ge A_j+1,
$$
and define
$$
        \alpha=[0;W_{N_1},d_1,W_{N_2},d_2,W_{N_3},d_3,\ldots].
$$
We now check the local terms of $\alpha$.  There are three types of cuts.

\textbf{Type 1: the central cut before the digit $1$.}
Consider the central digit in the block $W_{N_j}$.  By the choice of $N_j$, the corresponding one-sided tails satisfy $\beta_j+\gamma_j<1$.  Hence, the ordinary convergent immediately before this central digit is not admissible, since
$$
        \frac{1}{1+\beta_j+\gamma_j}
        >\frac12.
$$
Thus, the two ordinary gaps around this digit are replaced by one skipped-convergent contribution
$$
        G(\beta_j,\gamma_j)
        =\frac{1+\beta_j+\gamma_j}{4}.
$$
This value is within $\varepsilon_j$ of $\Lambda(s)$.  In particular, these central contributions tend to $\Lambda(s)$, and therefore
$$
        \mathfrak m(\alpha)\ge \Lambda(s).
$$
There is no additional cut immediately after the central digit $1$; it is part of the same skipped-convergent term.

\textbf{Type 2: cuts at a bounded distance from the central digit.}
In the block $W_{N_j}$, consider a non-central cut whose distance from the central digit $1$ is at most $R_j$.  Such a cut may have, outside the displayed central window, arbitrary continuations coming from the rest of the block, from separators, or from neighbouring blocks. By Lemma~\ref{lem:finite-window-stability}, this estimate is uniform over all possible continuations outside $W_{N_j}$. Hence, the $F$--value of the cut is at most
$$
        \Lambda(s)+\varepsilon_j.
$$
The reduction from the infinite two-sided model to finite blocks is contained entirely in Lemma~\ref{lem:finite-window-stability}.

\textbf{Type 3: cuts at a long distance from the central digit.}
All remaining cuts are far from the central digit of every block they meet. They can be controlled by just the first digit, independently of the rest of the structure.

We use the following elementary estimate.  Suppose that, after the change of variables
$$
        u=\frac{x}{1-x},
        \qquad
        v=\frac{y}{1-y},
$$
the two continued fractions defining $u$ and $v$ begin with digits at least
$A_j$.  Then $0<u,v\le\eta_j$, and hence
$$
        H(x,y)=\Kfun(u,v)
        \le u+v+u^2+v^2
        \le 2\eta_j+2\eta_j^2
        <4\varepsilon_j.
$$
Equivalently, the corresponding $F$--value is at most
$$
        \frac14+\varepsilon_j.
$$

Now take a type 3 cut inside a block $W_{N_j}$.  Since the cut is at a distance larger than $R_j$ from the central digit, it lies either between two $b$-digits with indices at least $R_j$, or between two $c$-digits with indices at least
$R_j$.  After the change of variables, the first digit on each side is one of these digits, possibly decreased by $1$.  By the choice of $R_j$, both first digits are still at least $A_j$.  Therefore, the preceding estimate gives an $F$--value at most $1/4+\varepsilon_j$.

The same argument applies to cuts adjacent to a separator $d_j$.  On one side, the first digit is $d_j$, possibly decreased by $1$, and on the other side, the first digit is one of the endpoint digits of $W_{N_j}$ or $W_{N_{j+1}}$, again possibly decreased by $1$.  These digits are at least $A_j$ for large $j$. Thus every separator cut with index $j$ has $F$--value at most
$$
        \frac14+\varepsilon_j.
$$
Since $\Lambda(s)>1/4$, these long-distance and separator cuts cannot affect the final limsup.

There are no skipped-convergent terms other than those in Type 1.  Indeed, all
partial quotients except the central digits $1$ are at least $2$, and then the
ordinary convergent at that position is admissible.  The central digits $1$ were
already treated in Type 1.

Combining the three types, the central skipped terms tend to $\Lambda(s)$, all
bounded-distance non-central cuts are bounded by $\Lambda(s)+\varepsilon_j$, and
all long-distance cuts are bounded by $1/4+\varepsilon_j$.  Since
$\varepsilon_j\to0$ and $\Lambda(s)>1/4$, we get
$$
        \mathfrak m(\alpha)=\Lambda(s).
$$
Thus $\Lambda(s)\in\MS$ for every irrational $s\in(0,1)$.

It remains to consider rational $s\in(0,1)$.  Choose irrational numbers
$s_j\in(0,1)$ with $s_j\to s$.  For each $j$, apply the irrational construction to $s_j$, but keep only one finite block chosen with accuracy $\varepsilon_j$. So each $W_{N_j}$ will once again have length $2N_j+1$, but in this case $W_{N_j}$ will be defined from the decomposition of $s_j$, meaning that each next $W_{N_j}$ will depend on a different number $s_j$.
Concatenate these selected blocks with separators tending to infinity.  The same three-type argument applies: the central terms tend to
$\Lambda(s_j)$, the bounded-distance terms are bounded by $\Lambda(s_j)+\varepsilon_j$, and the long-distance and separator terms tend to $1/4$.  Since $\Lambda(s_j)\to\Lambda(s)$, the resulting irrational continued
fraction $\alpha$ satisfies
$$
        \mathfrak m(\alpha)=\Lambda(s).
$$
Hence $\Lambda(s)\in\MS$ for every $0<s<1$.

We have proved $(1/4,1/2)\subset\MS$.  Moshchevitin proved the opposite bound $\MS\subset[1/4,1/2]$ and showed that both endpoints $1/4$ and $1/2$ are attained.  Consequently,
$$
        \MS=\left[\frac14,\frac12\right],
$$
as claimed.
\end{proof}

\section{Hausdorff dimension of level sets}

For $m\in(1/4,1/2)$ put
$$
        \Theta_m=\{\alpha\in(0,1)\setminus\Q:\mathfrak m(\alpha)=m\}.
$$

We shall use two standard facts from the dimension theory of continued fractions.  Let $U=(a_1,\ldots,a_n)$ be a finite word in positive integers.  We write $Q(U)$ for the denominator of $[0;a_1,\ldots,a_n]$, and
$$
        I(U)=\{[0;a_1,\ldots,a_n,a_{n+1},\ldots]:a_{n+1},a_{n+2},\ldots\in\N\}
$$
for the corresponding continued-fraction cylinder.  Then
\begin{equation}\label{eq:cf-cylinder-size}
        \frac{1}{2Q(U)^2}\le |I(U)|\le \frac{1}{Q(U)^2}.
\end{equation}
Moreover, if $U,V$ are finite words, then
\begin{equation}\label{eq:continuant-product}
        Q(U)Q(V)\le Q(U,V)\le 2Q(U)Q(V),
\end{equation}
where $(U,V)$ denotes concatenation.  These are the usual elementary estimates
for continued-fraction cylinders and continuants; see, for instance,
\cite[Chapter~1]{KhinchinCF}.

\begin{lemma}\label{lem:zero-density-insertion}
Let $\mathcal A=\{A,A+1,\ldots,B\}$ be a finite alphabet, and let
$$
        E_{\mathcal A}=\{[0;a_1,a_2,\ldots]:a_i\in\mathcal A\text{ for all }i\}.
$$
Let $W_j$ be fixed finite words and let $L_j\to\infty$.  Put
$$
        N_n=L_1+\cdots+L_n.
$$
Assume that
$$
        \sum_{j=1}^{n+1}\log Q(W_j)=o(N_n).
$$
For $x=[0;a_1,a_2,\ldots]\in E_{\mathcal A}$, cut the digit sequence into consecutive blocks $V_j$ of lengths $L_j$, and define
$$
        \Phi(x)=[0;W_1,V_1,W_2,V_2,W_3,V_3,\ldots].
$$
Then
$$
        \dim_H \Phi(E_{\mathcal A})\ge \dim_H E_{\mathcal A}.
$$
\end{lemma}

\begin{proof}
We give the standard distortion argument.  Let $x,y\in E_{\mathcal A}$, and suppose that their common continued-fraction prefix has length $k$, where $N_n\le k<N_{n+1}$.  The corresponding common prefix of $\Phi(x)$ and $\Phi(y)$ is obtained by inserting $W_1,\ldots,W_{n+1}$, possibly with $W_{n+1}$ only contributing through the fixed prefix before the next free digit.  By \eqref{eq:continuant-product} and \eqref{eq:cf-cylinder-size},
$$
        -\log |\Phi(x)-\Phi(y)|
        =
        -\log |x-y|+O\left(\sum_{j=1}^{n+1}\log Q(W_j)+n\right).
$$
Here $n=o(N_n)$ because $L_j\to\infty$, and $-\log |x-y|\asymp k\geq N_n$ on the finite-alphabet set $E_{\mathcal A}$.  Hence, for every $\eta>0$ and all sufficiently close $\Phi(x),\Phi(y)$,
$$
        |\Phi(x)-\Phi(y)|\ge C_\eta |x-y|^{1+\eta}
$$
with a constant $C_\eta>0$.  Therefore the inverse map $\Phi^{-1}:\Phi(E_{\mathcal A})\to E_{\mathcal A}$ is $1/(1+\eta)$-H\"older at small scales.  Consequently,
$$
        \dim_H E_{\mathcal A}
        \le
        (1+\eta)\dim_H \Phi(E_{\mathcal A}).
$$
Letting $\eta\downarrow0$ gives
$$
        \dim_H \Phi(E_{\mathcal A})\ge \dim_H E_{\mathcal A}.
$$
\end{proof}

Recall the notation
$$
        E_A=\{[0;a_1,a_2,\ldots]:a_i\ge A\text{ for all }i\},
        \qquad
        \delta_A=\dim_{H}E_A.
$$
We shall use the standard finite-alphabet approximation
\begin{equation}\label{eq:deltaA-finite-approx}
        \delta_A=\sup_{B\ge A}\dim_{H}E_{A,B},
        \qquad
        E_{A,B}=\{[0;a_1,a_2,\ldots]: A\le a_i\le B\text{ for all }i\},
\end{equation}
for countable conformal iterated function systems; see Good's classical work on restricted continued fractions \cite{Good1941} and the general finite-subsystem approximation theorem of Mauldin--Urba\'nski \cite{MauldinUrbanski1996}.

\begin{proof}[Proof of Theorem \ref{thm:large-digit-dimension}]
Put $s=4m-1$.  We first check the uniform estimate for inserted large digits.  Suppose that a cut has adjacent partial quotients at least $A$.  Then the two corresponding tails $x,y$ satisfy $x,y\le 1/A$.  After the change of variables
$$
        u=\frac{x}{1-x},
        \qquad
        v=\frac{y}{1-y},
$$
we have $0<u,v\le 1/(A-1)$.  Hence
$$
        H(x,y)=\Kfun(u,v)
        \le u+v+u^2+v^2
        \le \frac{2}{A-1}+\frac{2}{(A-1)^2}
        <s.
$$
By \eqref{F_to_H}, every such $F$--contribution is strictly smaller than $(1+s)/4=m$.

Fix $d<\delta_A$.  By \eqref{eq:deltaA-finite-approx}, choose $B\ge A$ such that
$$
        d<\dim_H E_{A,B},
$$
and put $\mathcal A=\{A,A+1,\ldots,B\}$.  Choose the control blocks $W_{N_j}$ as in the proof of Theorem \ref{thm:main}, with accuracies $\varepsilon_j\downarrow0$, and so that every cut of $W_{N_j}$ outside the central control window has adjacent partial quotients at least $A$; in particular, the endpoint partial quotients are at least $A$.  Thus, the skipped-convergent contribution produced by the central digit $1$ of $W_{N_j}$ lies in $(m-\varepsilon_j,m+\varepsilon_j)$, all non-central cuts in the central control window have $F$--value at most $m+\varepsilon_j$, and every remaining non-skipped cut inside $W_{N_j}$ has adjacent partial quotients at least $A$.  If $s\in\Q$, we take $W_{N_j}$ from the decomposition of irrational numbers $s_j\to s$, exactly as in the rational part of the proof of Theorem \ref{thm:main}.

After the blocks $W_{N_j}$ have been fixed, choose integers $L_j\to\infty$ so fast that
$$
        \sum_{i=1}^{n+1}\log Q(W_{N_i})=o(L_1+\cdots+L_n).
$$
For each $x\in E_{A,B}$, cut its digit sequence into consecutive blocks $V_j$ of lengths $L_j$, and define
$$
        \Phi(x)=[0;W_{N_1},V_1,W_{N_2},V_2,W_{N_3},V_3,\ldots].
$$
Let
$$
        \mathcal E=\Phi(E_{A,B}).
$$

We claim that $\mathcal E\subset\Theta_m$.  The central digit of each $W_{N_j}$ gives a skipped-convergent contribution tending to $m$.  The other cuts in the central control window of $W_{N_j}$ have limsup at most $m$.  Every cut inside a filler block $V_j$, every cut across the boundary between $W_{N_j}$ and $V_j$ or between $V_j$ and $W_{N_{j+1}}$, and every cut in $W_{N_j}$ outside the central control window has adjacent partial quotients at least $A$, and hence has $F$--value strictly smaller than $m$ by the uniform estimate above.  Since all partial quotients outside the distinguished central digits are at least $2$, no other skipped-convergent terms occur.  Therefore $\mathfrak m(\alpha)=m$ for every $\alpha\in\mathcal E$.

By Lemma \ref{lem:zero-density-insertion},
$$
        \dim_H\mathcal E
        \ge
        \dim_H E_{A,B}
        >d.
$$
Since $\mathcal E\subset\Theta_m$, we get $\dim_H\Theta_m\ge d$.  Letting $d\uparrow\delta_A$ gives
$$
        \dim_H\Theta_m\ge\delta_A.
$$
\end{proof}

\subsection{Upper bounds and the endpoint $m=1/4$}\label{sec:dimension}

We now record some upper bounds for the level sets $\Theta_m$. The first one shows that, except for the full-measure endpoint $m=1/2$, all level sets have dimension strictly smaller than one.

We shall use the following elementary dimension drop for one forbidden
continued-fraction block.

\begin{lemma}\label{lem:L2L-dimension-drop}
Fix $L\in\N$, and let
$$
X_L=
\{[0;a_1,a_2,\ldots]:(L,2,L)\text{ occurs only finitely many times in }
(a_i)_{i\ge1}\}.
$$
Then
$$
\dim_H X_L<1.
$$
\end{lemma}

\begin{proof}
Let
$$
Y_L=\{[0;a_1,a_2,\ldots]:(L,2,L)\text{ never occurs in }(a_i)_{i\ge1}\}.
$$
Every element of $X_L$ has a finite prefix after which the word $(L,2,L)$ never occurs. Hence
$$
X_L\subset \bigcup_U \Phi_U(Y_L),
$$
where $U$ ranges over all finite words and
$$
\Phi_U([0;a_1,a_2,\ldots])=[0;U,a_1,a_2,\ldots].
$$
Each $\Phi_U$ is Lipschitz. Thus, it is enough to show that
$$
\dim_H Y_L<1.
$$

Put
$$
W=(L,2,L).
$$
Group the continued-fraction digits into consecutive blocks of length $3$, beginning with the first digit, and define
$$
Z_L=
\{[0;A_1,A_2,\ldots]:A_j\in\mathbb N^3,\ A_j\ne W
\text{ for every }j\}.
$$
Then
$$
Y_L\subset Z_L,
$$
so it remains to prove that $\dim_H Z_L<1$.

For a finite word $U$, once again recall that $I(U)$ denotes the corresponding continued-fraction cylinder, and $Q(U)$ its continuant denominator. We shall use the standard cylinder estimates
$$
|I(U,b)|\ge \frac{1}{(b+1)^2}|I(U)|
$$
and
$$
\frac{|I(U,A)|}{|I(U)|}
\le
\frac{C}{Q(A)^2},
$$
where $C>0$ is absolute, which follows from \eqref{eq:cf-cylinder-size}, \eqref{eq:continuant-product}.

Applying the first estimate three times to $W=(L,2,L)$ gives
$$
|I(U,W)|\ge \zeta_L |I(U)|,
\qquad
\zeta_L:=\frac{1}{9(L+1)^4}.
$$

Fix a finite word $U$. The cylinders $I(U,A)$, $A\in\mathbb N^3$, partition $I(U)$ up to endpoints. Since the child corresponding to $A=W$ has length at least $\zeta_L|I(U)|$, we get
$$
\sum_{\substack{A\in\mathbb N^3\\ A\ne W}}
|I(U,A)|
\le
(1-\zeta_L)|I(U)|.
$$

Choose $s_0\in(1/2,1)$. For $A=(a,b,c)$, we have
$$
Q(A)=abc+a+c.
$$
Since $abc+a+c\ge abc$, the series
$$
\sum_{a,b,c\ge1} Q(a,b,c)^{-2s_0}
$$
converges. Hence we may choose $M\ge\max\{L,2\}$ so large that
$$
C_0
\sum_{\substack{A=(a,b,c)\in\mathbb N^3\\ \max(a,b,c)>M}}
Q(A)^{-2s_0}
<
\frac{\zeta_L}{4},
$$
where $C_0:=\max\{1,C\}$.

For the finitely many blocks $A=(a,b,c)$ with $1\le a,b,c\le M$, the lower cylinder estimate gives
$$
\frac{|I(U,A)|}{|I(U)|}
\ge
\lambda_M,
\qquad
\lambda_M:=(M+1)^{-6}.
$$
Choose $s\in(s_0,1)$ sufficiently close to $1$ so that
$$
\lambda_M^{s-1}(1-\zeta_L)
<
1-\frac{\zeta_L}{2}.
$$

For $A\in\{1,\ldots,M\}^3$, put
$$
r_A=\frac{|I(U,A)|}{|I(U)|}.
$$
Since $r_A\ge\lambda_M$ and $s<1$, we have
$$
r_A^s\le \lambda_M^{s-1}r_A.
$$
Therefore
\begin{align*}
\sum_{\substack{A\in\{1,\ldots,M\}^3\\ A\ne W}}
|I(U,A)|^s
&\le
|I(U)|^s\lambda_M^{s-1}
\sum_{\substack{A\in\mathbb N^3\\ A\ne W}} r_A \\
&\le
\lambda_M^{s-1}(1-\zeta_L)|I(U)|^s \\
&\le
\left(1-\frac{\zeta_L}{2}\right)|I(U)|^s.
\end{align*}

For the tail, since $s\ge s_0$, the upper cylinder estimate gives
\begin{align*}
\sum_{\substack{A=(a,b,c)\in\mathbb N^3\\ \max(a,b,c)>M}}
|I(U,A)|^s
&\le
|I(U)|^s
\sum_{\substack{A\in\mathbb N^3\\ \max(A)>M}}
\left(\frac{C}{q(A)^2}\right)^s \\
&\le
C_0|I(U)|^s
\sum_{\substack{A\in\mathbb N^3\\ \max(A)>M}}
q(A)^{-2s_0} \\
&<
\frac{\zeta_L}{4}|I(U)|^s.
\end{align*}
Combining the finite part and the tail, we obtain, uniformly in $U$,
$$
\sum_{\substack{A\in\mathbb N^3\\ A\ne W}}
|I(U,A)|^s
\le
\rho |I(U)|^s,
\qquad
\rho:=1-\frac{\zeta_L}{4}<1.
$$

Iterating over the allowed three-letter blocks gives
$$
\sum_{\substack{A_1,\ldots,A_n\in\mathbb N^3\\ A_j\ne W\ \forall j}}
|I(A_1,\ldots,A_n)|^s
\le
\rho^n.
$$
The cylinders appearing in this sum cover $Z_L$, and their diameters tend uniformly to zero as $n\to\infty$. Hence
$$
\mathcal H^s(Z_L)=0.
$$
Therefore
$$
\dim_H Z_L\le s<1.
$$
Since $Y_L\subset Z_L$, and $X_L$ is contained in a countable union of Lipschitz images of $Y_L$, we conclude that
$$
\dim_H X_L<1.
$$
\end{proof}

Now we can apply this lemma to get the following theorem.

\begin{theorem}\label{thm:upper-bound-nontrivial}
For every $m\in(1/4,1/2)$ one has
$$
\dim_H\Theta_m<1.
$$
\end{theorem}

\begin{proof}
Write
$$
s=4m-1\in(0,1).
$$
Choose an integer $L\ge4$ so large that
\begin{equation}\label{eq:L-for-upper-bound}
1-\frac{4}{L-1}>s.
\end{equation}
We claim that the word
$$
W=(L,2,L)
$$
can occur only finitely many times in the continued-fraction expansion of any
$\alpha\in\Theta_m$.

Indeed, suppose that $W$ occurs infinitely many times, so that we have the situation
$$
\ldots,L,2,L,\ldots .
$$
Consider the cut immediately before the digit $2$.  Since the two adjacent partial quotients are both at least $2$, this cut gives an ordinary $F$-term, not a skipped-convergent $G$-term.  Put
$$
x=[0;2,L,\ldots],
\qquad
y=[0;L,\ldots],
$$
and
$$
u=\frac{x}{1-x},
\qquad
v=\frac{y}{1-y}.
$$
Then
$$
u\ge \frac{L}{L+1}>1-\frac{1}{L-1},
\qquad
0<v\le \frac{1}{L-1}.
$$
Using
$$
\Kfun(u,v)
=
1-\frac{1-(u-v)^2}{2uv+u+v+1},
$$
we obtain
$$
\Kfun(u,v)
\ge (u-v)^2
\ge
\left(1-\frac{2}{L-1}\right)^2
\ge
1-\frac{4}{L-1}
>s.
$$
By \eqref{F_to_H}, the corresponding $F$-contribution is strictly larger than $(1+s)/4=m$.  Therefore, if the word $(L,2,L)$ occurred infinitely many times, the limsup defining $\mathfrak m(\alpha)$ would be strictly larger than $m$. This is impossible for $\alpha\in\Theta_m$.

Thus,
$$
\Theta_m\subset X_L,
$$
and the strict upper bound follows from Lemma \ref{lem:L2L-dimension-drop}.
\end{proof}

Next, we prove a more quantitative upper bound near the endpoint $1/4$.

For $R\ge2$, define
$$
\mathcal Y_R
=
\left\{
[0;a_1,a_2,\ldots]:
\max(a_n,a_{n+1},a_{n+2})>R
\text{ for all sufficiently large } n
\right\}.
$$
Thus, eventually, no block of three consecutive partial quotients is entirely
contained in $\{1,\ldots,R\}$.

\begin{lemma}\label{lem:small-triples-forbidden}
Let $R\ge2$ and let $m=(1+s)/4$, where $0\le s<1$.  If
\begin{equation}\label{eq:s-small-relative-R}
s<\frac{1}{R+1},
\end{equation}
then
$$
\Theta_m\subset \mathcal Y_R.
$$
\end{lemma}

\begin{proof}
Suppose that $\alpha=[0;a_1,a_2,\ldots]\in\Theta_m$ has infinitely many blocks of the form
$$
a_n,a_{n+1},a_{n+2}\le R.
$$
We show that each such block forces a local contribution strictly larger than
$m$.

First, consider a cut between two adjacent partial quotients both bounded by $R$.  If this cut gives an ordinary $F$-term, then the two one-sided tails $x,y$ satisfy
$$
\frac{x}{1-x}\ge \frac1R,
\qquad
\frac{y}{1-y}\ge \frac1R.
$$
Put $u=x/(1-x)$ and $v=y/(1-y)$.  We claim that, for all $u,v\ge1/R$,
\begin{equation}\label{eq:K-lower-small-tail}
\Kfun(u,v)\ge \frac1R.
\end{equation}
Indeed, writing $a=1/R\le1/2$, we have
\begin{align*}
&u^2+u+v^2+v
-a(2uv+u+v+1)   \\
&\qquad =
u^2+v^2-2auv+(1-a)(u+v)-a                  \\
&\qquad \ge
(1-a)(u^2+v^2)+(1-a)(u+v)-a                 \\
&\qquad \ge
2a(1-a)-a
=
a(1-2a)
\ge0.
\end{align*}
Dividing by the positive denominator $2uv+u+v+1$ gives
\eqref{eq:K-lower-small-tail}.  Hence, the corresponding $F$-term satisfies
$$
4F(x,y)-1=H(x,y)=\Kfun(u,v)\ge\frac1R>\frac1{R+1}>s,
$$
so $F(x,y)>m$.

It remains to consider the case in which the relevant adjacent cut is swallowed by a skipped-convergent term.  This can only happen around a digit $1$.  Since we have a block of three consecutive partial quotients all at most $R$, at least one of the two tails entering the corresponding $G$-term begins with a partial quotient at most $R$.  Therefore, for that skipped term,
$$
4G(x,y)-1=x+y\ge \frac1{R+1}>s,
$$
and again $G(x,y)>m$.

Thus, every occurrence of a block of three consecutive partial quotients bounded by $R$ produces a local contribution strictly larger than $m$.  Infinitely many such blocks would imply $\mathfrak m(\alpha)>m$, contradicting $\alpha\in\Theta_m$.  Hence, only finitely many such blocks occur, which is exactly $\alpha\in\mathcal Y_R$.
\end{proof}

\begin{proposition}\label{prop:YR-dimension-bound}
For every $R\ge2$,
$$
\dim_H\mathcal Y_R
\le
\inf\left\{
t>\frac12:
\sum_{\max(a,b,c)>R}(abc)^{-2t}<1
\right\}.
$$
In particular,
$$
\lim_{R\to\infty}\dim_H\mathcal Y_R=\frac12.
$$
\end{proposition}

\begin{proof}
Fix $t>1/2$ and put
$$
\Sigma_R(t)
=
\sum_{\max(a,b,c)>R}(abc)^{-2t}.
$$
Since $\sum_{a\ge1}a^{-2t}<\infty$, we have
$$
\Sigma_R(t)\to0
\qquad (R\to\infty).
$$
Assume that $\Sigma_R(t)<1$.

Finite prefixes do not affect Hausdorff dimension, so it is enough to cover the set of numbers for which every disjoint block of three sufficiently late digits has a maximum larger than $R$.  Let
$$
U=(a_1,\ldots,a_{3n})
$$
be a word split into triples
$$
(a_{3j+1},a_{3j+2},a_{3j+3}),
\qquad 0\le j\le n-1,
$$
each of which has a maximum larger than $R$.  Since
$$
Q(U)\ge \prod_{i=1}^{3n}a_i,
$$
the cylinder estimate \eqref{eq:cf-cylinder-size} gives
$$
|I(U)|^t
\le
\prod_{i=1}^{3n} a_i^{-2t}.
$$
Therefore, the total $t$-volume of all such cylinders of order $3n$ is bounded by
$$
\left(
\sum_{\max(a,b,c)>R}(abc)^{-2t}
\right)^n
=
\Sigma_R(t)^n,
$$
which tends to zero.  Hence $\dim_H\mathcal Y_R\le t$.  Taking the infimum over
such $t$ proves the first assertion.

For the second assertion, let $t>1/2$.  Then $\Sigma_R(t)<1$ for all sufficiently
large $R$, and so
$$
\limsup_{R\to\infty}\dim_H\mathcal Y_R\le t.
$$
Letting $t\downarrow1/2$ gives
$$
\limsup_{R\to\infty}\dim_H\mathcal Y_R\le\frac12.
$$
The reverse inequality is not needed below, but follow,s for instan,ce from the sets with all partial quotients tending to infinity.  Thus, the limiting upper bound is $1/2$.
\end{proof}

We can now determine the Hausdorff dimension of the endpoint level set.

\begin{theorem}\label{thm:endpoint-quarter}
One has
$$
\dim_H\Theta_{1/4}=\frac12.
$$
\end{theorem}

\begin{proof}
Let
$$
E_\infty=\{[0;a_1,a_2,\ldots]:a_n\to\infty\}.
$$
By a classical theorem of Good \cite{Good1941},
$$
\dim_H E_\infty=\frac12.
$$
We first show that
$$
E_\infty\subset\Theta_{1/4}.
$$
Indeed, if $a_n\to\infty$, then eventually all partial quotients are at least $2$, so there are no skipped-convergent $G$-terms from some point on.  Moreover, at every ordinary cut the two one-sided tails tend to zero, and therefore, after Lemma \ref{lem:K},
$$
H(x,y)=\Kfun\left(\frac{x}{1-x},\frac{y}{1-y}\right)\to0.
$$
Equivalently,
$$
F(x,y)\to\frac14.
$$
Since Moshchevitin's theorem gives $\mathfrak m(\alpha)\ge1/4$ for every irrational $\alpha$, we get $\mathfrak m(\alpha)=1/4$.  Hence
$E_\infty\subset\Theta_{1/4}$, and consequently
$$
\dim_H\Theta_{1/4}\ge\frac12.
$$

For the upper bound, fix $t>1/2$.  Choose $R\ge2$ so large that
$$
\sum_{\max(a,b,c)>R}(abc)^{-2t}<1.
$$
Lemma \ref{lem:small-triples-forbidden}, applied with $s=0$, gives
$$
\Theta_{1/4}\subset \mathcal Y_R.
$$
Therefore Proposition \ref{prop:YR-dimension-bound} implies
$$
\dim_H\Theta_{1/4}\le t.
$$
Letting $t\downarrow1/2$ gives the desired upper bound.
\end{proof}

\begin{corollary}\label{cor:near-quarter-upper}
One has
$$
\lim_{m\downarrow1/4}\dim_H\Theta_m=\frac12.
$$
More precisely, if $R\ge2$ and
$$
0\le 4m-1<\frac1{R+1},
$$
then
$$
\dim_H\Theta_m
\le
\inf\left\{
t>\frac12:
\sum_{\max(a,b,c)>R}(abc)^{-2t}<1
\right\}.
$$
\end{corollary}

\begin{proof}
The quantitative upper bound follows immediately from Lemma \ref{lem:small-triples-forbidden} and Proposition \ref{prop:YR-dimension-bound}.  Now fix $t>1/2$.  Choose $R$ so large that
$$
\sum_{\max(a,b,c)>R}(abc)^{-2t}<1.
$$
Then for all $m>1/4$ sufficiently close to $1/4$, namely for
$4m-1<1/(R+1)$, we have
$$
\dim_H\Theta_m\le t.
$$
Thus
$$
\limsup_{m\downarrow1/4}\dim_H\Theta_m\le\frac12.
$$
The reverse inequality follows from Theorem \ref{thm:large-digit-dimension}, which gives $\dim_H\Theta_m\ge1/2$ for every $m\in(1/4,1/2)$. Therefore, the limit exists and is equal to $1/2$.
\end{proof}

\begin{remark}\label{rem:theta-quarter-not-only-large-digits}
The inclusion
$$
E_\infty\subset\Theta_{1/4}
$$
is a proper one.  For example, if $A_j\to_{j\to\infty}\infty$ sufficiently fast and
$$
\alpha=[0;A_1,1,A_2,1,A_3,1,\ldots],
$$
then $a_n(\alpha)$ does not tend to infinity, but every skipped-convergent contribution around a digit $1$ is of the form
$$
\frac{1+x_j+y_j}{4},
\qquad
x_j,y_j\to0,
$$
and all remaining ordinary contributions also tend to $1/4$.  Hence
$\alpha\in\Theta_{1/4}\setminus E_\infty$.
\end{remark}

\section*{Acknowledgements}
The author thanks Prof. Moshchevitin for stating the original problem about the Minkowski spectrum at the Simons Semesters event in Będlewo, which has led to this work. 

This work was partially supported by the Simons Foundation grant (award no. SFI-MPS-T-Institutes-00010825) and from State Treasury funds as part of a task commissioned by the Minister of Science and Higher Education under the project “Organization of the Simons Semesters at the Banach Center - New Energies in 2026-2028” (agreement no. MNiSW/2025/DAP/491).

\appendix

\section{Proof of Lemmas \ref{lem:arithmetic-bound}, \ref{lem:left-rectangle}, \ref{lem:right-square}}\label{sec:addendix}

\begin{proof}[Proof of Lemma \ref{lem:arithmetic-bound}]
We prove the first assertion first.  By Lemma \ref{lem:endpoint-principle}, it
is enough to check the four corners of the rectangle
$[0,U]\times[1/b_1,V_{b_1}]$.  At the two corners with $u=0$, we have
$$
        \Kfun(0,v)=v\le V_{b_1}<L,
$$
where the last inequality follows from
$$
        L-V_{b_1}
        =\frac1{c_1+1}-\frac1{b_1(b_1c_1+b_1-1)}>0.
$$
Thus, it remains to prove
\begin{equation}\label{eq:corner-U-b}
        \Kfun\left(U,\frac{1}{b_1}\right)\le L,
\end{equation}
and
\begin{equation}\label{eq:corner-U-Vb}
        \Kfun(U,V_{b_1})\le L.
\end{equation}

The exceptional case $(b_1,c_1)=(2,2)$ is immediate:
$$
        L=\frac56,
        \qquad
        U=1,
        \qquad
        V_{b_1}=\frac35,
$$
and a direct substitution gives
$$
        \Kfun\left(1,\frac12\right)=\frac{11}{14}<\frac56,
        \qquad
        \Kfun\left(1,\frac35\right)=\frac{74}{95}<\frac56.
$$
Since $\Kfun$ is symmetric and $V_{c_1}=V_{b_1}=3/5$ in this exceptional case,
the second range \eqref{eq:arithmetic-second-range} follows from the same corner checks.  Assume now that
$(b_1,c_1)\ne(2,2)$, so that the displayed formula for $U$ is valid.

For $b_1=2$ and $c_1\geq3$, after substituting $U=4/(c_1(c_1+1)-6)$, inequalities
\eqref{eq:corner-U-b} and \eqref{eq:corner-U-Vb} reduce respectively to
$$
        \frac{3c_1^4+6c_1^3-17c_1^2-52c_1-20}{(c_1-2)(c_1+1)(c_1+3)(3c_1^2+3c_1-2)}>0
$$
and
$$
        \frac{(3c_1^2+5c_1+8)(3c_1^4+4c_1^3-17c_1^2-38c_1-16)}{2c_1^2(c_1+1)(c_1+3)(c_1-2)(2c_1+1)(3c_1+5)}>0,
$$
which holds for all $c_1\ge3$.

For $b_1=3$ and $c_1\geq6$, the substitution $U=9/(c_1(c_1+1)-12)$ reduces \eqref{eq:corner-U-b} and \eqref{eq:corner-U-Vb} respectively to
$$
        \frac{4c_1^4-4c_1^3-71c_1^2-162c_1-63}{(c_1-3)(c_1+1)(c_1+4)(2c_1-1)(2c_1+3)}>0
$$
and
$$
        \frac{2(4c_1^2+7c_1+9)(4c_1^4-7c_1^3-68c_1^2-120c_1-54)}{3c_1^2(c_1-3)(c_1+1)(c_1+4)(3c_1+2)(4c_1+7)}>0,
$$
which holds for all $c_1\geq6$.

It remains to treat $b_1\ge4$.  Put
$$
        c_1=b_1(b_1-1)+T,
        \qquad T\ge0.
$$
Then \eqref{eq:corner-U-b} and \eqref{eq:corner-U-Vb} reduce respectively to
$$
        \frac{P_1(T)}{D_1(T)}>0,
        \qquad
        \frac{P_2(T)}{D_2(T)}>0,
$$
where
\begin{align*}
P_1(T)={}&b_1(b_1+1)T^4
        +b_1(b_1+1)(3b_1^2-2b_1+2)T^3\\
        &+b_1(3b_1^5-3b_1+1)T^2
        +b_1^2(b_1^6+b_1^5-3b_1^4-b_1^3-3b_1^2-b_1-2)T\\
        &+b_1^6(b_1^3-3b_1^2+2b_1-3),\\[3pt]
D_1(T)={}&b_1(T+b_1^2+1)(T+b_1^2-2b_1)(T+b_1^2-b_1+1)\\
        &\times\bigl((b_1+1)T^2+(2b_1^3-b_1+1)T+b_1^5-b_1^4+b_1^2-2b_1\bigr),
\end{align*}
and
\begin{align*}
P_2(T)={}&(b_1-1)(b_1+1)^2T^6
        +(b_1-1)(b_1+1)^2(5b_1^2-4b_1+3)T^5\\
        &+(10b_1^7-5b_1^6-10b_1^5+14b_1^4-8b_1^3-3b_1^2+9b_1-3)T^4\\
        &+(10b_1^9-10b_1^8-10b_1^7+26b_1^6-30b_1^5+16b_1^4+11b_1^3\\
        &\qquad -13b_1^2+11b_1-1)T^3\\
        &+b_1(b_1-1)(5b_1^9-15b_1^7+24b_1^6-33b_1^5+18b_1^4-9b_1^2\\
        &\qquad +10b_1-4)T^2\\
        &+b_1^2(b_1^2+1)(b_1^9+b_1^8-17b_1^7+39b_1^6-49b_1^5+37b_1^4\\
        &\qquad -12b_1^3-10b_1^2+9b_1-5)T\\
        &+b_1^3(b_1^2+1)^2(b_1^3-2b_1^2+b_1+1)(b_1^4-4b_1^3+5b_1^2-5b_1+2),\\[3pt]
D_2(T)={}&b_1(T+b_1^2+1)(T+b_1^2-2b_1)(T+b_1^2-b_1)^2(T+b_1^2-b_1+1)\\
        &\times\bigl((b_1+1)T+b_1^3+b_1+1\bigr)(b_1T+b_1^3-b_1^2+b_1-1).
\end{align*}
All factors in $D_1(T)$ and $D_2(T)$ are positive for $b_1\ge4$ and $T\ge0$, and all displayed coefficients in $P_1(T)$ and $P_2(T)$ are positive for $b_1\ge4$.  Hence \eqref{eq:corner-U-b} and \eqref{eq:corner-U-Vb} follow.

For the second range \eqref{eq:arithmetic-second-range}, Lemma \ref{lem:endpoint-principle} reduces the proof
to the four corners of $[1/c_1,V_{c_1}]\times[0,U]$.  At the two corners with
$v=0$, $\Kfun(u,0)=u\le V_{c_1}<L$.  The last inequality follows because, after writing $c_1=b_1(b_1-1)+T$, it becomes
$$
        \frac{(b_1+1)T^2+(2b_1^3-b_1)T+b_1^5-b_1^4-b_1^2-2b_1-1}{b_1(T+b_1^2-b_1+1)((b_1+1)T+b_1^3-b_1-1)}>0,
$$
which is positive for $b_1\ge2$ and $T\ge0$.
Thus, it remains to prove
\begin{equation}\label{eq:corner-Vc-U}
        \Kfun\left(\frac{1}{c_1},U\right)\le L,
        \qquad
        \Kfun(V_{c_1},U)\le L.
\end{equation}
For $b_1=2$, $c_1\geq3$, the two inequalities in \eqref{eq:corner-Vc-U} reduce respectively to
$$
        \frac{(c_1+2)(c_1^6+2c_1^5-16c_1^4+2c_1^3+43c_1^2-60c_1-36)}{2c_1(c_1-2)(c_1+1)(c_1+3)(c_1^3+2c_1^2-c_1+2)}>0
$$
and
$$
        \frac{9c_1^7+30c_1^6-134c_1^5-208c_1^4+621c_1^3+114c_1^2-1504c_1-272}{2(c_1-2)(c_1+1)(c_1+3)(3c_1-1)(3c_1^3+5c_1^2-4c_1+8)}>0,
$$
which holds for all $c_1\ge3$.
For $b_1=3$, $c_1\geq6$, the two inequalities in \eqref{eq:corner-Vc-U} reduce respectively to
$$
        \frac{c_1^7+4c_1^6-39c_1^5-68c_1^4+208c_1^3-90c_1^2-1080c_1-432}{3c_1(c_1-3)(c_1+1)(c_1+4)(c_1^3+2c_1^2-2c_1+6)}>0
$$
and
$$
        \frac{16c_1^7+56c_1^6-663c_1^5-806c_1^4+4031c_1^3-2265c_1^2-17226c_1-3807}{3(c_1-3)(c_1+1)(c_1+4)(4c_1-1)(4c_1^3+7c_1^2-9c_1+27)}>0,
$$
which holds for all $c_1\ge6$.
Finally let $b_1\ge4$, and put $c_1=b_1(b_1-1)+T$.  The two inequalities in \eqref{eq:corner-Vc-U} reduce respectively to
$$
        \frac{P_3(T)}{D_3(T)}>0,
        \qquad
        \frac{P_4(T)}{D_4(T)}>0,
$$
where
\begin{align*}
P_3(T)={}&T^7+(7b_1^2-7b_1+4)T^6+(21b_1^4-43b_1^3+44b_1^2-27b_1+6)T^5\\
        &+(35b_1^6-110b_1^5+165b_1^4-166b_1^3+103b_1^2-39b_1+4)T^4\\
        &+(b_1-1)(35b_1^7-115b_1^6+185b_1^5-219b_1^4+163b_1^3-90b_1^2\\
        &\qquad +24b_1-1)T^3\\
        &+b_1(21b_1^9-115b_1^8+290b_1^7-486b_1^6+598b_1^5-546b_1^4+380b_1^3\\
        &\qquad -188b_1^2+60b_1-6)T^2\\
        &+b_1^2(7b_1^{10}-47b_1^9+144b_1^8-289b_1^7+432b_1^6-494b_1^5\\
        &\qquad +445b_1^4-318b_1^3+170b_1^2-66b_1+12)T\\
        &+b_1^3(b_1^{11}-8b_1^{10}+29b_1^9-68b_1^8+119b_1^7-162b_1^6\\
        &\qquad +176b_1^5-158b_1^4+113b_1^3-66b_1^2+28b_1-8),\\[3pt]
D_3(T)={}&b_1(T+b_1^2+1)(T+b_1^2-2b_1)(T+b_1^2-b_1)(T+b_1^2-b_1+1)\\
        &\times\bigl(T^3+(3b_1^2-3b_1+2)T^2+(3b_1^4-6b_1^3+7b_1^2-5b_1+1)T\\
        &\qquad +b_1^6-3b_1^5+5b_1^4-6b_1^3+5b_1^2-2b_1\bigr),
\end{align*}
and
\begin{align*}
P_4(T)={}&(b_1+1)^2T^7+(b_1+1)(7b_1^3-3b_1+2)T^6\\
        &+b_1(21b_1^5-b_1^4-21b_1^3+6b_1^2-5b_1-11)T^5\\
        &+(35b_1^8-40b_1^7-20b_1^6+24b_1^5-37b_1^4+13b_1^2-7b_1-2)T^4\\
        &+(35b_1^{10}-80b_1^9+35b_1^8+6b_1^7-48b_1^6+55b_1^5+4b_1^4\\
        &\qquad -b_1^3+19b_1^2+5b_1-1)T^3\\
        &+b_1(21b_1^{11}-73b_1^{10}+81b_1^9-51b_1^8+8b_1^7+54b_1^6\\
        &\qquad -24b_1^5+19b_1^4+5b_1^3-11b_1^2+4)T^2\\
        &+b_1^2(7b_1^{12}-33b_1^{11}+57b_1^{10}-60b_1^9+49b_1^8-7b_1^7\\
        &\qquad -7b_1^6+5b_1^5-17b_1^4-11b_1^3-3b_1^2-7b_1-5)T\\
        &+b_1^3(b_1^{13}-6b_1^{12}+14b_1^{11}-20b_1^{10}+23b_1^9-17b_1^8\\
        &\qquad +9b_1^7-9b_1^6-2b_1^5-4b_1^4+b_1^3+5b_1+2),\\[3pt]
D_4(T)={}&b_1(T+b_1^2+1)(T+b_1^2-2b_1)(T+b_1^2-b_1+1)((b_1+1)T+b_1^3-b_1-1)\\
        &\times\bigl((b_1+1)T^3+(3b_1^3-b_1+1)T^2+(3b_1^5-3b_1^4+b_1^3-2b_1)T\\
        &\qquad +b_1^7-2b_1^6+2b_1^5-2b_1^4+b_1^3+b_1^2\bigr).
\end{align*}
All factors in $D_3(T)$ and $D_4(T)$ are positive for $b_1\ge4$ and $T\ge0$, and all displayed coefficients in $P_3(T)$ and $P_4(T)$ are positive for $b_1\ge4$.  Hence the two
inequalities in \eqref{eq:corner-Vc-U} hold.  This proves \eqref{eq:arithmetic-second-range} and completes the lemma.

\end{proof}

\begin{proof}[Proof of Lemma \ref{lem:left-rectangle}]
By Lemma \ref{lem:endpoint-principle}, the maximum on the rectangle is attained
at one of the four corners.  Thus, it is enough to check
$$
        \Kfun(0,0),\quad
        \Kfun(1/b_1,0),\quad
        \Kfun(0,U_L),\quad
        \Kfun(1/b_1,U_L).
$$
The first two values are $0$ and $1/b_1$, and are at most $L$.

If $(b_1,c_1)=(2,2)$, then $L=5/6$, $U_L=3/4$, and
$$
        \Kfun(0,3/4)=\frac34<\frac56,
        \qquad
        \Kfun(1/2,3/4)=\frac{11}{16}<\frac56.
$$

Assume now that $(b_1,c_1)\ne(2,2)$, so
$$
        U_L=\frac{b_1^2}{c_1(c_1+1)-b_1-b_1^2}.
$$
First, $U_L\le L$.  For $b_1=2$, $c_1\geq3$, this becomes
$$
        \frac{c_1^3+4c_1^2-11c_1-26}{2(c_1-2)(c_1+1)(c_1+3)}>0,
$$
which holds for all $c_1\ge3$.
For $b_1\ge3$, write $c_1=b_1(b_1-1)+T$, $T\ge0$.  Then $U_L\le L$ becomes
$$
        \frac{P_5(T)}{D_5(T)}>0,
$$
where
\begin{align*}
P_5(T)={}&T^3+(3b_1^2-2b_1+2)T^2
        +(3b_1^4-5b_1^3+4b_1^2-4b_1+1)T\\
        &+b_1(b_1^5-3b_1^4+3b_1^3-5b_1^2+b_1-2),\\[3pt]
D_5(T)={}&b_1(T+b_1^2+1)(T+b_1^2-2b_1)(T+b_1^2-b_1+1).
\end{align*}
All factors in $D_5(T)$ and all displayed coefficients in $P_5(T)$ are positive for $b_1\ge3$ and $T\ge0$.

It remains to check $\Kfun(1/b_1,U_L)\le L$.  For $b_1=2$, $c_1\geq3$, this becomes
$$
        \frac{3c_1^4+6c_1^3-17c_1^2-52c_1-20}{(c_1-2)(c_1+1)(c_1+3)(3c_1^2+3c_1-2)}>0,
$$
which holds for all $c_1\ge3$.
For $b_1\ge3$, again put $c_1=b_1(b_1-1)+T$.  Then $\Kfun(1/b_1,U_L)\le L$ becomes
$$
        \frac{P_6(T)}{D_6(T)}>0,
$$
where
\begin{align*}
P_6(T)={}&(b_1+1)T^4+(b_1+1)(3b_1^2-2b_1+2)T^3+(3b_1^5-3b_1+1)T^2\\
        &+b_1(b_1^6+b_1^5-3b_1^4-b_1^3-3b_1^2-b_1-2)T\\
        &+b_1^5(b_1^3-3b_1^2+2b_1-3),\\[3pt]
D_6(T)={}&(T+b_1^2+1)(T+b_1^2-2b_1)(T+b_1^2-b_1+1)\\
        &\times\bigl((b_1+1)T^2+(2b_1^3-b_1+1)T+b_1^5-b_1^4+b_1^2-2b_1\bigr).
\end{align*}
All factors in $D_6(T)$ and all displayed coefficients in $P_6(T)$ are positive for $b_1\ge3$ and $T\ge0$.  Hence
$\Kfun(1/b_1,U_L)\le L$, and the lemma follows.

\end{proof}

\begin{proof}[Proof of Lemma \ref{lem:right-square}]
By Lemma \ref{lem:endpoint-principle}, the maximum on the square is attained at
a corner.  Since
$$
        \Kfun(0,0)=0,
        \qquad
        \Kfun(W,0)=\Kfun(0,W)=W,
$$
it is enough to prove
$$
        W\le L,
        \qquad
        \Kfun(W,W)\le L.
$$

If $(b_1,c_1)=(2,2)$, then $L=5/6$, $W=1/2$, and
$$
        \Kfun(1/2,1/2)=\frac35<\frac56.
$$

Assume $(b_1,c_1)\ne(2,2)$.  For $b_1=2$, $c_1\geq3$, the inequality $W\le L$ becomes
$$
        \frac{c_1^3+4c_1^2-7c_1-14}{2(c_1-1)(c_1+1)(c_1+2)}>0,
$$
which holds for all $c_1\ge3$.
For $b_1\ge3$, putting $c_1=b_1(b_1-1)+T$, $T\ge0$, the inequality $W\le L$ becomes
$$
        \frac{P_7(T)}{D_7(T)}>0,
$$
where
\begin{align*}
P_7(T)={}&T^3+(3b_1^2-2b_1+2)T^2
        +(3b_1^4-5b_1^3+5b_1^2-4b_1+1)T\\
        &+b_1(b_1^5-3b_1^4+4b_1^3-5b_1^2+2b_1-2),\\[3pt]
D_7(T)={}&b_1(T+b_1^2-b_1+1)\\
        &\times\bigl(T^2+(2b_1^2-2b_1+1)T+b_1^4-2b_1^3+2b_1^2-2b_1\bigr).
\end{align*}
All factors in $D_7(T)$ and all displayed coefficients in $P_7(T)$ are positive for $b_1\ge3$ and $T\ge0$.

It remains to prove $\Kfun(W,W)\le L$.  For $b_1=2$, $c_1\geq3$, this becomes
$$
        \frac{(c_1-1)(c_1^4+6c_1^3+c_1^2-12c_1-28)}{2(c_1+1)(c_1^4+2c_1^3+5c_1^2+4c_1+20)}>0,
$$
which holds for all $c_1\ge3$.
For $b_1\ge3$, put $c_1=b_1(b_1-1)+T$.  Then $\Kfun(W,W)\le L$ becomes
$$
        \frac{P_8(T)}{D_8(T)}>0,
$$
where
\begin{align*}
P_8(T)={}&T^5+(5b_1^2-4b_1+3)T^4
        +(10b_1^4-18b_1^3+20b_1^2-12b_1+3)T^3\\
        &+(10b_1^6-30b_1^5+48b_1^4-48b_1^3+29b_1^2-12b_1+1)T^2\\
        &+b_1(5b_1^7-22b_1^6+48b_1^5-68b_1^4+64b_1^3-42b_1^2+18b_1-4)T\\
        &+b_1^2(b_1^8-6b_1^7+17b_1^6-32b_1^5+42b_1^4-38b_1^3+26b_1^2-12b_1+4),\\[3pt]
D_8(T)={}&b_1(T+b_1^2-b_1+1)\\
        &\times\bigl(T^4+(4b_1^2-4b_1+2)T^3+(6b_1^4-12b_1^3+14b_1^2-8b_1+1)T^2\\
        &\qquad +(4b_1^6-12b_1^5+22b_1^4-24b_1^3+14b_1^2-4b_1)T\\
        &\qquad +b_1^8-4b_1^7+10b_1^6-16b_1^5+18b_1^4-12b_1^3+4b_1^2\bigr).
\end{align*}
All factors in $D_8(T)$ and all displayed coefficients in $P_8(T)$ are positive for $b_1\ge3$ and $T\ge0$.  Therefore
$\Kfun(W,W)\le L$, and the lemma follows.

\end{proof}

\bibliographystyle{abbrv}
\bibliography{bibliog}

\end{document}